\documentclass[12pt,a4paper,pdftex]{amsart}
\usepackage[margin=2.6cm]{geometry}
\usepackage[T1]{fontenc}
\usepackage[utf8]{inputenc}
\usepackage{lmodern}
\usepackage{amssymb,mathtools}
\usepackage{tikz-cd}
\numberwithin{equation}{section}

\newtheorem{thm}{Theorem}[section]
\newtheorem{prop}[thm]{Proposition}
\newtheorem{lem}[thm]{Lemma}
\newtheorem{cor}[thm]{Corollary}

\theoremstyle{definition}
\newtheorem{defn}[thm]{Definition}
\newtheorem{assump}[thm]{Assumption}
\theoremstyle{remark}

\DeclareMathOperator{\Hom}{Hom}

\newcommand*{\id}{\mathrm{id}}

\newcommand{\BSbimod}{\mathcal{BS}}

\newcommand*{\sqbinom}[2]{\genfrac{[}{]}{0pt}{}{#1}{#2}}

\DeclarePairedDelimiter{\floor}{\lfloor}{\rfloor}
\newcommand{\Z}{\mathbb{Z}}
\newcommand{\Coef}{\mathbb{K}}

\title{A homomorphism between Bott-Samelson bimodules}
\author{Noriyuki Abe}
\address{Graduate School of Mathematical Sciences, the University of Tokyo, 3-8-1 Komaba, Meguro-ku, Tokyo 153-8914, Japan.}
\email{abenori@ms.u-tokyo.ac.jp}
\begin{document}
\begin{abstract}
In the previous paper, we defined a new category which categorifies the Hecke algebra.
This is a generalization of the theory of Soergel bimodules.
To prove theorems, the existences of certain homomorphisms between Bott-Samelson bimodules are assumed.
In this paper, we prove this assumption.
We only assume the vanishing of certain two-colored quantum binomial coefficients.
\end{abstract}
\maketitle
\section{Introduction}
In recent development of representation theory of algebraic reductive groups, the Hecke category plays central role.
Here, the Hecke category means a categorification of the Hecke algebra of Coxeter groups.
One can find the importance of the Hecke category in representation theory in Williamson's survey~\cite{MR3966750}.

There are several incarnations of the Hecke category.
They can be roughly divided into two types: geometric ones and combinatorial ones.
The geometric Hecke category which appeared in representation theory first is the category of semisimple perverse shaves on the flag variety.
This category is the Hecke category with a field of characteristic zero.
Juteau-Mauter-Williamson~\cite{MR3230821} introduced the notion of parity sheave.
The category of parity sheaves on the flag variety is a geometric incarnation of the Hecke category with any field.
When the characteristic of the ground field is zero, parity sheaves are the same as semisimple perverse shaves.

Soergel~\cite{MR2329762} introduced a category which is now called the category of Soergel bimodules.
Similar to the situation of the geometric ones, if the characteristic of the ground field is zero, this category is the Hecke category.
Soergel's category is equivalent to the category of semisimple perverse sheaves on the flag variety over a filed of characteristic zero.
This fact is used to prove the Kosuzl duality of the category $\mathcal{O}$~\cite{MR1322847}.

Soergel's category does not behave well over a field of positive characteristic in general.
As a generalization of Soergel's category, the author introduced a new combinatorial category and proved that this category is the Hecke category in more general situation than Soergel's theory~\cite{arXiv:1901.02336_accepted}.
There is also another combinatorial category defined by Elias-Williamson~\cite{MR3555156} which is defined earlier than~\cite{arXiv:1901.02336_accepted}.
The category is called the diagrammatic Hecke category and it is proved that the category is the Hecke category in general situation.
We remark that these categories are equivalent to each others when they behave well~\cite{MR3805034,arXiv:1901.02336_accepted,arXiv:2004.09014}.

It is proved that these theories works well very general, including most cases over a field of positive characteristic.
However, we still need some assumptions.
The situation is subtle.
In \cite{arXiv:1901.02336_accepted}, we need one non-trivial assumption which we recall later.
One problem is that this assumption is not easy to check.
In \cite{arXiv:1901.02336_accepted}, a sufficient condition for this assumption which we can check easier is given.
However the author thought that the assumption holds in more general.
The aim of this paper is to prove this assumption under a mild condition.
(The situation is also subtle for the diagrammatic Hecke category. See \cite[5.1]{arXiv:2011.05432}. We do not discuss about it in this paper.)

\subsection{Soergel bimodules}
We recall the category introduced in \cite{arXiv:1901.02336_accepted} and the assumption.
Let $(W,S)$ be a Coxeter system such that $\#S < \infty$ and $\Coef$ a commutative integral domain.
We fix a realization~\cite[Definition~3.1]{MR3555156} $(V,\{\alpha_s\}_{s\in S},\{\alpha_s^\vee\}_{s\in S})$ of $(W,S)$ over $\Coef$.
Namely $V$ is a free $\Coef$-module of finite rank with an action of $W$, $\alpha_s\in V$, $\alpha_s^\vee\in\Hom_{\Coef}(V,\Coef)$ such that
\begin{enumerate}
\item $s(v) = v - \langle \alpha_s^\vee,v\rangle\alpha_s$ for any $s\in S$ and $v\in V$.
\item $\langle \alpha_s^\vee,\alpha_s\rangle = 2$.
\item Let $s,t\in S$ ($s\ne t$) and $m_{s,t}$ the order of $st$.
If $m_{s,t} < \infty$ then the two-colored quantum numbers $[m_{s,t}]_X,[m_{s,t}]_Y$ attached to $\{s,t\}$ are both zero. (See \ref{subsec:Finite Coxeter group of rank two and a realization} for the definition of these numbers.)
\end{enumerate}
We also assume the Demazure surjectivity, namely we assume that $\alpha_s\colon \Hom_{\Coef}(V,\Coef)\to \Coef$ and $\alpha_s^\vee\colon V\to \Coef$ are both surjective for any $s\in S$.

We define the category $\mathcal{C}$ as follows.
Let $R = S(V)$ be the symmetric algebra and $Q$ the field of fractions of $R$.
An object of $\mathcal{C}$ is $(M,(M_Q^x)_{x\in W})$ such that $M$ is a graded $R$-bimodule, $M_Q^x$ is a $Q$-bimodule such that $mp = x(p)m$ for any $m\in M_Q^x$, $p\in Q$ and $M\otimes_{R}Q = \bigoplus_{x\in W}M_Q^x$.
We also assume that $M$ is flat as a right $R$-module.
A morphism $\varphi\colon (M,(M_Q^x)_{x\in W})\to (N,(N_Q^x)_{x\in W})$ is an $R$-bimodule homomorphism $\varphi\colon M\to N$ of degree zero such that $(\varphi\otimes\id_Q)(M_Q^x)\subset N_Q^x$.
We often write $M$ for $(M,(M_Q^x))$.
For $M,N\in \mathcal{C}$, we define $M\otimes N\in \mathcal{C}$ as follows.
As an $R$-bimodule, we have $M\otimes N = M\otimes_{R}N$ and $(M\otimes N)_Q^x = \bigoplus_{yz = x}M_Q^y\otimes_Q N^z_Q$.

For each $s\in S$, we have an object denoted by $B_s$.
As a graded $R$-bimodule, $B_s = R\otimes_{R^s}R(1)$ where $(1)$ is the grading shift and $R^s = \{f\in R\mid s(f) = f\}$.
Then $B_s$ has a unique lift in $\mathcal{C}$ such that $(B_s)_Q^x = 0$ unless $x = e,s$.
An object of a form
\[
B_{s_1}\otimes B_{s_2}\otimes \cdots \otimes B_{s_l}(n)
\]
for $s_1,\ldots,s_l\in S$ and $n\in \Z$ is called a \emph{Bott-Samelson bimodule}.
Let $\BSbimod$ denote the category of Bott-Samelson bimodules.

In \cite{arXiv:1901.02336_accepted}, we proved that $\BSbimod$ gives a categorification of the Hecke algebra assuming the following.
We refer it as \cite[Assumption~3.2]{arXiv:1901.02336_accepted}.
\begin{quote}
Let $s,t\in S$, $s\ne t$ such that $m_{s,t}$ is finite.
Then there exists a morphism
\[
\overbrace{B_s\otimes B_t\otimes\cdots}^{m_{s,t}}
\to
\overbrace{B_t\otimes B_s\otimes\cdots}^{m_{s,t}}
\]
which sends $(1\otimes 1)\otimes (1\otimes 1)\otimes \cdots\otimes(1\otimes 1)$ to $(1\otimes 1)\otimes (1\otimes 1)\otimes \cdots\otimes(1\otimes 1)$.
\end{quote}

We introduce the following assumption.

\begin{assump}\label{assump:Assumption in Introduction}
For any $s,t\in S$ such that $m_{s,t} < \infty$, the two-colored quantum binomial coefficients $\sqbinom{m_{s,t}}{k}_X$ and $\sqbinom{m_{s,t}}{k}_Y$ are both zero for any $k = 1,\ldots, m_{s,t} - 1$.
\end{assump}
For the definition of two-colored quantum binomial coefficients, see \ref{subsec:two-colored quantum numbers} and \ref{subsec:Soergel bimodules}.
This assumption is related to the existence of Jones-Wenzl projectors. (See Proposition~\ref{prop:the assumptions} and \cite[Conjecture~6.23]{arXiv:2011.05432}.)
The main theorem of this paper is the following.

\begin{thm}[Theorem~\ref{thm:main theorem}]
Under Assumption~\ref{assump:Assumption in Introduction}, \cite[Assumption~3.2]{arXiv:1901.02336_accepted} holds.
\end{thm}

Note that Assumption~\ref{assump:Assumption in Introduction} is a very mild condition.
For example, if a realization comes from a root system then it is always satisfied (Proposition~\ref{prop:assumption, in root system}).

\subsection{Diagrammatic category}
Let $\mathcal{D}$ be the diagrammatic Hecke category defined in \cite{MR3555156}.
We assume that the category $\mathcal{D}$ is ``well-defined''~\cite[5.1]{arXiv:2011.05432}.
In \cite{MR3555156}, under some assumptions~\cite[5.3]{arXiv:2011.05432}, a functor $\mathcal{F}$ from $\mathcal{D}$ to $\BSbimod$ is constructed.
The construction of $\mathcal{F}$ is deeply related to \cite[Assumption~3.2]{arXiv:1901.02336_accepted} as we explain here.

The morphisms in the category $\mathcal{D}$ are defined by generators and relations.
So to define $\mathcal{F}$, we have to define the images of generators.
Except the generators called $2m_{s,t}$-valent vertices ($s,t\in S$), the images of generators are given easily.
For $2m_{s,t}$-valent vertices, the images should be morphisms in \cite[Assumption~3.2]{arXiv:1901.02336_accepted}.
Hence, to prove \cite[Assumption~3.2]{arXiv:1901.02336_accepted} is almost equivalent to the construction of $\mathcal{F}$.
Therefore as a consequence of our main theorem, we can prove the following.
\begin{thm}
Under Assumption~\ref{assump:Assumption in Introduction}, the category $\mathcal{D}$ is equivalent to $\BSbimod$.
\end{thm}

\subsection{Localized calculus}
In the proof, we use localized calculus.
Ideas of localized calculus are found in \cite{MR3555156,arXiv:1901.02336_accepted} and more systematic treatment recently appeared in \cite{arXiv:2011.05432}.

Let $\mathcal{C}_Q$ be the category of $(P^x)_{x\in W}$ where $P^x$ is a $Q$-bimodule such that $mp = x(p)m$ for $p\in Q$ and $m\in P^x$.
A morphism $(P_1^x)_{x\in W}\to (P_2^x)_{x\in W}$ is $(\varphi_x)_{x\in W}$ where $\varphi_x\colon P_1^x\to P_2^x$ is a $Q$-bimodule homomorphism for any $x\in W$.
Then for $M\in \mathcal{C}$, $(M^x_Q)_{x\in W}\in \mathcal{C}_Q$.
We denote this object by $M_Q$.
For $M,N\in \mathcal{C}$ and a morphism $\varphi\colon M\to N$, we have a morphism $\varphi_Q\colon M_Q\to N_Q$.
Conversely, assume that $\varphi_Q\colon M_Q\to N_Q$ is given and if $\varphi_Q$ sends $M\subset M\otimes_R Q = \bigoplus_{x\in W}M_Q^x$ to $N$, then the restriction of $\varphi_Q$ to $M$ gives a morphism $M\to N$ in $\mathcal{C}$.

Let $M,N$ be two Bott-Samelson bimodules in \cite[Assumption~3.2]{arXiv:1901.02336_accepted}.
A candidate of $\varphi_Q\colon M_Q\to N_Q$ is given in \cite{arXiv:2011.05432}.
The hardest part is to prove that $\varphi_Q$ sends $M$ to $N$.

We check that $\varphi_Q$ gives a desired homomorphism by calculations.
One of the things which we need to prove is the following.
Let $s,t\in S$ such that $m_{s,t} < \infty$.
For simplicity, assume that $V$ is balanced, namely $[m_{s,t} - 1]_X = [m_{s,t} - 1]_Y = 1$.
Let $s_1\cdots s_{m_{s,t}}$ be a reduced expression of the longest element in the group $\langle s,t\rangle$ generated by $\{s,t\}$.
Then for any $g\in \langle s,t\rangle$, we have
\begin{equation}\label{eq:a in Introduction}
\sum_{e = (e_i)\in \{0,1\}^{m_{s,t}},s_1^{e_1}\cdots s_{m_{s,t}}^{e_{m_{s,t}}} = g}
\prod_{i = 1}^{m_{s,t}}s_1^{e_1}\cdots s_{i - 1}^{e_{i - 1}}\left(\frac{1}{\alpha_{s_i}}\right)
 = \frac{1}{\displaystyle\prod_{i = 1}^{m_{s,t}}s_1\cdots s_{i - 1}(\alpha_{s_{i}})}.
\end{equation}
(If $V$ comes from a root system, then this formula can be proved by applying the localization formula to the Bott-Samelson resolution of the flag variety. The author learned this from Syu Kato.)

In Section~\ref{sec:A calculation in the universal Coxeter system of rank two}, we calculate the left hand side of \ref{eq:a in Introduction}.
Moreover, we give an explicit formula of the left hand side for any sequence $(s_1,s_2,\ldots)$ of $\{s,t\}$.
The hardest part of this calculation is to find a correct result.
Once we find the correct formulation, the proof is done by induction.

For a general element $m\in M$, we first give a formula to express $\varphi_Q(m)$ using the left hand side of \eqref{eq:a in Introduction} (with any $s_1,s_2,\ldots$).
We also have an algorithm to check $\varphi_Q(m)\in N$ (Lemma~\ref{lem:criterion in M}).
In Section~\ref{sec:A homomorphism between Bott-Samelson bimodules}, using this algorithm and an explicit formula obtained in Section~\ref{sec:A calculation in the universal Coxeter system of rank two}, we prove the main theorem.

\subsection{On Assumption~\ref{assump:Assumption in Introduction}}
In \cite{arXiv:1901.02336_accepted}, a sufficient condition for \cite[Assumption~3.2]{arXiv:1901.02336_accepted} was given.
In \cite{MR3555156}, a sufficient condition for the existence of $\mathcal{F}$ was given.
Both conditions are stronger than Assumption~\ref{assump:Assumption in Introduction}.
It was expected that these theorems are proved under the weaker condition \cite[Remark~5.6]{arXiv:2011.05432} but concrete conditions were not known.

In this paper, we prove these theorems under Assumption~\ref{assump:Assumption in Introduction}.
Moreover, we prove that the theorems are almost equivalent to Assumption~\ref{assump:Assumption in Introduction}.
More precisely, we prove the following.
Let $\varphi_Q\colon M_Q\to N_Q$ be the morphism in $\mathcal{C}_Q$ introduced above and $\psi_Q\colon  N_Q\to M_Q$ the morphism obtaining by the same way as $\varphi_Q$.
Then $\varphi_Q$ and $\psi_Q$ give desired morphisms if and only if Assumption~\ref{assump:Assumption in Introduction} holds (Proposition~\ref{prop:1 -> 1}).
Therefore the author thinks that Assumption~\ref{assump:Assumption in Introduction} is the final form in this direction

\subsection*{Acknowledgments}
The author thanks Syu Kato for giving many helpful comments.
The author was supported by JSPS KAKENHI Grant Number 18H01107.

\section{A calculation in the universal Coxeter system of rank two}\label{sec:A calculation in the universal Coxeter system of rank two}
Since our main theorem is concerned with a rank two Coxeter system, in almost all part of this paper, we only consider a Coxeter system of rank two.
In this section, we give an explicit formula of the left hand side of \eqref{eq:a in Introduction}.
Such formula can be proved in a universal form.
Hence we work with the universal Coxeter system of rank two in this section.

\subsection{Two-colored quantum numbers}\label{subsec:two-colored quantum numbers}
In this subsection we introduce two-colored quantum numbers \cite{MR3462556,MR3555156}.
Let $\Z[X,Y]$ be the polynomial ring with two variables over $\Z$.
\begin{defn}[{two-colored quantum numbers, \cite[Definition~3.6]{MR3555156}}]
For $n\in\Z_{\ge 0}$, we define $[n]_X,[n]_Y\in \Z[X,Y]$ by
\begin{align*}
& [0]_X = [0]_Y = 0,\quad [1]_X = [1]_Y = 1,\\
& [n + 1]_X = X[n]_Y - [n - 1]_X,\\
& [n + 1]_Y = Y[n]_Y - [n - 1]_Y.
\end{align*}
\end{defn}
Note that $[2]_X = X$ and $[2]_Y = Y$.
Define $\sigma\colon \{X,Y\}\to \{X,Y\}$ by $\sigma(X) = Y$ and $\sigma(Y) = X$.
Then for $Z\in \{X,Y\}$, we have
\[
[n + 1]_Z = [2]_Z[n]_{\sigma(Z)} - [n - 1]_Z.
\]

We prove some properties of these polynomials which we will use later.
Some of them are known well or immediately follow from known results.
We give proofs for the sake of completeness.

\begin{lem}\label{lem:quantum num, even, odd}
Let $n\in\Z_{\ge 0}$.
\begin{enumerate}
\item If $n$ is odd, then $[n]_X = [n]_Y$.\label{lem:enum:quantum num, odd}
\item If $n$ is even, $[n]_X/X,[n]_Y/Y\in \Z[X,Y]$ and $[n]_X/X = [n]_Y/Y$.\label{lem:enum:quantum num, even}
\item We have $[n]_{Z} = [n]_{\sigma^n(Z)}$ for $Z\in \{X,Y\}$.
\item We have $[n]_X,[n]_Y\ne 0$ if $n > 0$.\label{lem:enum:quantum num, shift}
\end{enumerate}
\end{lem}
\begin{proof}
The first two statements follow from the definition using induction.
For the third, if $n$ is odd then it follows from (1).
If $n$ is even then it is obvious.
We also have $[n]_X(2,2) = [n]_Y(2,2) = n$ which follows easily by induction.
Hence $[n]_X,[n]_Y\ne 0$.
\end{proof}

An obvious consequence of (\ref{lem:enum:quantum num, odd}) (\ref{lem:enum:quantum num, even}) which will be used several times in this paper is the following.
For $k_1,\ldots,k_r,l_1,\ldots,l_s\in\Z_{> 0}$ such that $\#(2\Z\cap \{k_1,\ldots,k_r\}) = \#(2\Z\cap \{l_1,\ldots,l_s\})$, then $([k_1]_X\cdots [k_r]_X)/([l_1]_X\cdots [l_s]_X) = ([k_1]_Y\cdots [k_r]_Y)/([l_1]_Y\cdots [l_s]_Y)$.

\begin{lem}\label{lem:quantum, (n + 1)(m + 1)-nm=n+m+1}
Let $m,n\in\Z_{\ge 0}$ and $Z\in \{X,Y\}$.
Then we have
\[
[m + n + 1]_{\sigma^n(Z)} = [m + 1]_{Z}[n + 1]_{\sigma(Z)} - [m]_{\sigma(Z)}[n]_{Z}.
\]
\end{lem}
\begin{proof}
We prove by induction on $n$.
The cases of $n = 0$ and $n = 1$ follow from the definitions.
Assume that the lemma holds for $n - 1,n - 2$.
Then 
\begin{align*}
[m + n + 1]_{\sigma^{n}(Z)} & = [2]_{\sigma^n(Z)}[m + n]_{\sigma^{n + 1}(Z)} - [m + n - 1]_{\sigma^n(Z)}\\
& = [2]_{\sigma^n(Z)}([m + 1]_{Z}[n]_{\sigma(Z)} - [m]_{\sigma(Z)}[n - 1]_{Z}) \\
& \quad - ([m + 1]_{Z}[n - 1]_{\sigma(Z)} - [m]_{\sigma(Z)}[n - 2]_{Z})\\
& = [m + 1]_{Z}([2]_{\sigma^n(Z)}[n]_{\sigma(Z)} - [n - 1]_{\sigma(Z)})\\
& \quad - [m]_{\sigma(Z)}([2]_{\sigma^n(Z)}[n - 1]_{Z} - [n - 2]_{Z}).
\end{align*}
By Lemma~\ref{lem:quantum num, even, odd} (\ref{lem:enum:quantum num, shift}), we have $[n]_{\sigma(Z)} = [n]_{\sigma^{n + 1}(Z)}$ and $[n - 1]_{\sigma(Z)} = [n - 1]_{\sigma^{n}(Z)}$.
Hence $[2]_{\sigma^n(Z)}[n]_{\sigma(Z)} - [n - 1]_{Z} = [2]_{\sigma^n(Z)}[n]_{\sigma^{n + 1}(Z)} - [n - 1]_{\sigma^{n}(Z)} = [n + 1]_{\sigma^n(Z)} = [n + 1]_{\sigma(Z)}$.
In the last we used Lemma~\ref{lem:quantum num, even, odd} (\ref{lem:enum:quantum num, shift}) again.
Similarly we have $[2]_{\sigma^n(Z)}[n - 1]_{Z} - [n - 2]_{Z} = [2]_{\sigma^n(Z)}[n - 1]_{\sigma^{n + 1}(Z)} - [n - 2]_{\sigma^n(Z)} = [n]_{\sigma^n(Z)} = [n]_Z$.
\end{proof}

\begin{lem}\label{lem:quantum, (m + n + 1)n - (m + n)(n  +1) = -m}
Let $m,n\in \Z_{\ge 0}$ and $Z\in \{X,Y\}$.
Then we have
\[
[m]_{\sigma^{n}(Z)} = [m + n]_{Z}[n + 1]_{\sigma(Z)} - [m + n + 1]_{\sigma(Z)}[n]_{Z}.
\]
\end{lem}
\begin{proof}
By the previous lemma, we have $[m + n + 1]_{\sigma(Z)}[n]_{Z} = [m + 2n]_{\sigma^{n}(Z)} + [m + n]_{Z}[n - 1]_{\sigma(Z)}$.
By swapping $m + n$ with $n$, we have $[m + n]_{Z}[n + 1]_{\sigma(Z)} = [m + 2n]_{\sigma^{m + n}(Z)} + [m + n - 1]_{\sigma(Z)}[n]_{Z}$.
By Lemma~\ref{lem:quantum num, even, odd} (\ref{lem:enum:quantum num, shift}), we have $[m + 2n]_{\sigma^{m + n}(Z)} = [m + 2n]_{\sigma^{n}(Z)}$.
Hence $[m + n]_{Z}[n + 1]_{\sigma(Z)} - [m + n + 1]_{\sigma(Z)}[n]_{Z} = [m + n - 1]_{\sigma(Z)}[n]_{Z} - [m + n]_{Z}[n - 1]_{\sigma(Z)}$.
Therefore, by induction on $n$, $[m + n]_{Z}[n + 1]_{\sigma(Z)} - [m + n + 1]_{\sigma(Z)}[n]_{Z} = [m]_{\sigma^{n}(Z)}[1]_{\sigma^{n + 1}(Z)} - [m + 1]_{\sigma^{n + 1}(Z)}[0]_{\sigma^{n}(Z)} = [m]_{\sigma^{n}(Z)}$.
\end{proof}
\begin{lem}\label{lem:quantum, (m + n + 1)(m + n) - (m + 1)m = n(2m + n + 1)}
For $m,n\in\Z_{\ge 0}$ and $Z\in \{X,Y\}$, we have 
\begin{align*}
[m + n + 1]_{\sigma^n(Z)}[m + n]_Z - [m + 1]_Z[m]_{\sigma^n(Z)}
 &=[n]_Z[2m + n + 1]_{\sigma^{m + 1}(Z)}\\
 & = [n]_{\sigma^n(Z)}[2m + n + 1]_{\sigma^{n + m}(Z)},\\
 [m + n + 1]_{\sigma^{n + 1}(Z)}[m + n + 1]_Z - [m]_Z[m]_{\sigma^{n + 1}(Z)}
 & = [n + 1]_{Z}[2m + n + 1]_{\sigma^{m}(Z)}
\end{align*}
\end{lem}
\begin{proof}
Applying Lemma~\ref{lem:quantum, (n + 1)(m + 1)-nm=n+m+1} to $[m + n + 1]_{\sigma^n(Z)}$ (resp.\ Lemma~\ref{lem:quantum, (m + n + 1)n - (m + n)(n  +1) = -m} to $[m]_{\sigma^n(Z)}$), we have
\begin{align*}
& [m + n + 1]_{\sigma^n(Z)}[m + n]_Z - [m + 1]_Z[m]_{\sigma^n(Z)}\\
& = ([m + 1]_{Z}[n + 1]_{\sigma(Z)} - [m]_{\sigma(Z)}[n]_Z)[m + n]_Z\\
& \quad- [m + 1]_Z([m + n]_Z[n + 1]_{\sigma(Z)} - [m + n + 1]_{\sigma(Z)}[n]_Z)\\
& = -[m]_{\sigma(Z)}[n]_Z[m + n]_Z + [m + 1]_Z[m + n + 1]_{\sigma(Z)}[n]_Z\\
& = [n]_Z([m + n + 1]_{\sigma(Z)}[m + 1]_Z - [m + n]_Z[m]_{\sigma(Z)}).
\end{align*}
The first formula of the lemma follows from Lemma~\ref{lem:quantum, (n + 1)(m + 1)-nm=n+m+1}.
The second follows from the first and Lemma~\ref{lem:quantum num, even, odd} (\ref{lem:enum:quantum num, shift}).
The third formula follows from a similar calculation.
\end{proof}

For $m,n\in \Z_{\ge 0}$ such that $n\le m$ and $Z\in \{X,Y\}$, define the two-colored quantum binomial coefficient $\sqbinom{m}{n}_Z$ \cite[Definition~6.7]{arXiv:2011.05432} by
\[
\sqbinom{m}{n}_Z = \frac{[m]_Z[m - 1]_Z\cdots [m - n + 1]_Z}{[n]_Z[n - 1]_Z \cdots [1]_Z}.
\]
By the lemma (2) below and induction, we have $\sqbinom{m}{n}_Z\in \Z[X,Y]$.

\begin{lem}\label{lem:an inductive formula on binomial}
Let $m,n\in\Z$ such that $1\le n\le m$ and $Z\in \{X,Y\}$.
\begin{enumerate}
\item 
\[
\sqbinom{m}{n}_{Z} = \sqbinom{m + 1}{n}_{\sigma^n(Z)}[n + 1]_Z - \sqbinom{m}{n - 1}_Z[m + 2]_{\sigma^{n + 1}(Z)}.
\]
\item 
\[
\sqbinom{m + 1}{n}_Z = \sqbinom{m}{n}_{\sigma^n(Z)}[n + 1]_{Z} - \sqbinom{m}{n - 1}_{Z}[m - n]_{\sigma^{n + 1}(Z)}.
\]
\end{enumerate}
\end{lem}
\begin{proof}
(1)
By Lemma~\ref{lem:quantum, (m + n + 1)n - (m + n)(n  +1) = -m}, we have $[m - n + 1]_{Z} = [m + 1]_{\sigma^{n}(Z)}[n + 1]_{\sigma^{n + 1}(Z)} - [m + 2]_{\sigma^{n + 1}(Z)}[n]_{\sigma^{n}(Z)}$.
By Lemma~\ref{lem:quantum num, even, odd} (\ref{lem:enum:quantum num, shift}), we have $[n]_{\sigma^{n}(Z)} = [n]_Z$ and $[n + 1]_{\sigma^{n + 1}(Z)} = [n + 1]_Z$.
Therefore 
\begin{align*}
\sqbinom{m}{n}_{Z}
& =
\frac{[m]_Z\cdots [m - n + 2]_Z}{[n]_Z\cdots [1]_Z}([m + 1]_{\sigma^{n}(Z)}[n + 1]_{Z} - [m + 2]_{\sigma^{n + 1}(Z)}[n]_{Z})\\
& = \frac{[m + 1]_{\sigma^{n}(Z)}[m]_Z\cdots [m - n + 2]_Z}{[n]_Z\cdots [1]_Z}[n + 1]_{Z}
- \sqbinom{m}{n - 1}_Z[m + 2]_{\sigma^{n + 1}(Z)}.
\end{align*}
Therefore it is sufficient to prove 
\[
\frac{[m + 1]_{\sigma^{n}(Z)}[m]_Z\cdots [m - n + 2]_Z}{[n]_Z\cdots [1]_Z}
=
\frac{[m + 1]_{\sigma^{n}(Z)}[m]_{\sigma^n(Z)}\cdots [m - n + 2]_{\sigma^{n}(Z)}}{[n]_{\sigma^n(Z)}\cdots [1]_{\sigma^n(Z)}}.
\]
If $n$ is even then we have nothing to prove.
If $n$ is odd, then $\#(2\Z\cap \{m,\ldots,m - n + 2\}) = \#(2\Z\cap \{n,\ldots,1\})$.
Hence it follows from Lemma~\ref{lem:quantum num, even, odd}.

\noindent
(2)
By Lemma~\ref{lem:quantum, (n + 1)(m + 1)-nm=n+m+1}, we have $[m + 1]_{Z} = [m - n + 1]_{\sigma^n(Z)}[n + 1]_{\sigma^{n + 1}(Z)} - [m - n]_{\sigma^{n + 1}(Z)}[n]_{\sigma^n(Z)}$.
By Lemma~\ref{lem:quantum num, even, odd} (\ref{lem:enum:quantum num, shift}), we have $[n + 1]_{\sigma^{n + 1}(Z)} = [n + 1]_{Z}$ and $[n]_{\sigma^n(Z)} = [n]_Z$.
Hence
\begin{align*}
\sqbinom{m + 1}{n}_Z & = \frac{[m]_Z\cdots [m - n + 2]_Z}{[n]_Z\cdots [n - 1]_Z\cdots [1]_Z}[m + 1]_Z\\
& = \frac{[m]_Z\cdots [m - n + 2]_Z[m - n + 1]_{Z}}{[n]_Z\cdots [1]_Z}[n + 1]_{Z} + \sqbinom{m}{n - 1}_Z[m - n]_{\sigma^{n + 1}(Z)}.
\end{align*}
It is sufficient to prove
\[
\frac{[m]_Z\cdots [m - n + 2]_Z[m - n + 1]_{\sigma^n(Z)}}{[n]_Z\cdots [1]_Z}
=
\sqbinom{m}{n}_{\sigma^n(Z)}.
\]
If $n$ is even we have nothing to prove.
If $n$ is odd, then we have $\#(2\Z\cap \{m,m - 1,\ldots,m- n + 2\}) = \#(2\Z\cap \{n,\ldots,1\})$.
Hence we get (2) by Lemma~\ref{lem:quantum num, even, odd}.
\end{proof}

\begin{lem}\label{lem:quantim binom}
We have
\[
\frac{[2m + n + 1]_{\sigma^{m + 1}(Z)}}{[m]_{\sigma^n(Z)}}\sqbinom{2m + n}{m - 1}_{\sigma^n(Z)}
=
\sqbinom{2m + n + 1}{m}_{\sigma^{n + 1}(Z)}.
\]
\end{lem}
\begin{proof}
Replacing $Z$ with $\sigma^n(Z)$, the lemma is equivalent to
\[
\frac{[2m + n + 1]_{\sigma^{m + n + 1}(Z)}[2m + n]_Z\cdots [m + n + 2]_Z}{[m]_Z[m - 1]_Z\cdots [1]_Z}
=
\frac{[2m + n + 1]_{\sigma(Z)}\cdots [m + n + 2]_{\sigma(Z)}}{[m]_{\sigma(Z)}[m - 1]_{\sigma(Z)}\cdots [1]_{\sigma(Z)}}.
\]
If $m + n + 1$ is even, then we have $\#(2\Z\cap \{2m + n + 1,\ldots,m + n + 2\}) = \#(2\Z\cap \{m,\ldots,1\})$.
Hence the lemma follows from Lemma~\ref{lem:quantum num, even, odd}.
If $m + n + 1$ is odd, then $\sigma^{m + n + 1}(Z) = \sigma(Z)$ and $\#(2\Z\cap \{2m + n,\ldots,m + n + 2\}) = \#(2\Z\cap \{m,\ldots,1\})$.
Hence again the lemma follows from Lemma~\ref{lem:quantum num, even, odd}.
\end{proof}

\subsection{A formula}
Let $(W,S)$ be the universal Coxeter system of rank two, namely the group $W$ is generated by the set of two elements $S = \{s,t\}$ and defined by relations $s^2 = t^2 = 1$.
The length function is denoted by $\ell$ and the Bruhat order is denoted by $\le$.
Let $V = \Z[X,Y]\alpha_s\oplus \Z[X,Y]\alpha_t$ be the free $\Z[X,Y]$-module of rank two with a basis $\{\alpha_s,\alpha_t\}$.
We define an action of $W$ on $V$ by 
\[
s(\alpha_s) = -\alpha_s,\quad s(\alpha_t) = \alpha_t + X\alpha_s,\quad t(\alpha_s) = \alpha_s + Y\alpha_t,\quad t(\alpha_t) = -\alpha_t.
\]
Let $\Phi = \{w(\alpha_s),w(\alpha_t)\mid w\in W\}$ be the set of roots and the set of positive roots $\Phi^+$ is defined by $\Phi^+ = \{w(\alpha_s)\mid ws > w\}\cup\{w(\alpha_t)\mid wt > w\}$.
For each $\alpha\in\Phi$, we have the reflection $s_{\alpha}\in W$.
This is defined as $s_{\alpha_s} = s$, $s_{\alpha_t} = t$, $s_{w(\alpha)} = ws_{\alpha}w^{-1}$ for $\alpha\in \{\alpha_s,\alpha_t\}$ and $w\in W$.

The following formula can be proved by induction.
\begin{lem}
We have
\[
(st)^k\alpha_s = [2k + 1]_X\alpha_s + [2k]_Y\alpha_t,\quad
(ts)^k\alpha_t = [2k]_X\alpha_s + [2k + 1]_Y\alpha_t,
\]
\end{lem}

\begin{lem}\label{lem:formula of root}
Let $\gamma\in \Phi^+$ and $g = s_{\gamma}$.
\begin{enumerate}
\item If $sg > g$, then 
\[
\gamma = \left[\frac{\ell(g) - 1}{2}\right]_X\alpha_s + \left[\frac{\ell(g) + 1}{2}\right]_Y\alpha_t.
\]
\item If $tg > g$, then 
\[
\gamma = \left[\frac{\ell(g) + 1}{2}\right]_X\alpha_s + \left[\frac{\ell(g) - 1}{2}\right]_Y\alpha_t.
\]
\end{enumerate}
\end{lem}
\begin{proof}
We have $\gamma = (ts)^k(\alpha_t)$ or $t(st)^k(\alpha_s)$ or $(st)^k(\alpha_s)$ or $s(ts)^k(\alpha_t)$.
If $\gamma = (ts)^k(\alpha_t)$, then $sg > g$ and $\ell(g) = 4k + 1$.
The lemma follows from the previous lemma.
If $\gamma = t(st)^k(\alpha_s)$, then $sg > g$ and $\ell(g) = 4k + 3$.
We have 
\begin{align*}
\gamma & = t([2k + 1]_X\alpha_s + [2k]_Y\alpha_t) = [2k + 1]_X(\alpha_s + Y\alpha_t) - [2k]_Y\alpha_t\\
& = [2k + 1]_X\alpha_s + ([2k + 1]_XY - [2k]_Y)\alpha_t = [2k + 1]_X\alpha_s + [2k + 2]_Y\alpha_t 
\end{align*}
and the lemma follows.
The proof of the other cases are similar.
\end{proof}

We define some elements which will be needed for our main formula.
We use the following notation for sequences in $S$.
A sequence in $S$ will be written with the underline like $\underline{w} = (s_1,\ldots,s_l)$.
We write $w = s_1\cdots s_l$.
For $u\in S$, put $(\underline{w},u) = (s_1,\ldots,s_l,u)$.
For $e = (e_1,\ldots,e_l)\in \{0,1\}^l$, we put $\underline{w}^e = s_1^{e_1}\cdots s_l^{e_l}$.
We set $\ell(\underline{w}) = l$.

For $g,w\in W$, we put 
\[
X_g^w =  \{\alpha\in\Phi^+\mid s_{\alpha}g\le w\}.
\]
Let $\underline{w} = (s_1,\ldots,s_l) \in S^l$ be a sequence of elements in $S$ and $g\in W$.
For a real number $r$, let $\floor{r}$ be the integral part of $r$.
We define $k_g^{\underline{w}}\in \Z[X,Y]$ as follows.
If $s_i = s_{i + 1}$ for some $i$ or $g\not\le w$, then $k_g^{\underline{w}} = 0$.
If $\ell(\underline{w}) = 0$, then $k_1^{\underline{w}} = 1$ and $k_g^{\underline{w}} = 0$ if $g\ne 1$.
Otherwise we put
\[
k_g^{\underline{w}}
=
\begin{cases}
\sqbinom{\ell(\underline{w}) - 1}{\floor*{\frac{\ell(\underline{w}) - \ell(g) - 1}{2}}}_{\sigma^{\ell(\underline{w}) - 1}(Z)} & (s_1g > g),\\
\sqbinom{\ell(\underline{w}) - 1}{\floor*{\frac{\ell(\underline{w}) - \ell(g)}{2}}}_{\sigma^{\ell(\underline{w}) - 1}(Z)} & (s_1g < g),
\end{cases}
\]
where $Z = X$ if $s_1 = s$ and $Z = Y$ if $s_1 = t$.

Let $R$ be the symmetric algebra of $V$ and $R^{\emptyset} = \Phi^{-1}R$ the ring of fractions.
We define an element $a^{\underline{w}}(g)$ of $R^{\emptyset}$ by
\[
a^{\underline{w}}(g)
=
\sum_{\underline{w}^e = g}
\frac{1}{\alpha_{s_1}}s_1^{e_1}\left(\frac{1}{\alpha_{s_2}}s_2^{e_2}\left(\cdots \frac{1}{\alpha_{s_{l - 1}}}s_{l - 1}^{e_{l - 1}}\left(\frac{1}{\alpha_{s_l}}\right)\cdots\right)\right)
=
\sum_{\underline{w}^e = g}
\prod_{i = 1}^l
(s_1^{e_1}\cdots s_{i - 1}^{e_{i - 1}})\left(\frac{1}{\alpha_{s_i}}\right)
.
\]
\begin{lem}
If $s_i = s_{i - 1}$ for some $i$, namely if $\underline{w}$ is not a reduced expression, then $a^{\underline{w}}(g) = 0$.
\end{lem}
\begin{proof}
Set $A = \{e\in \{0,1\}^l\mid \underline{w}^e = g\}$.
Define $f\colon A\to A$ by $f(e) = (e'_1,\ldots,e'_l)$ where $e'_i = 1 - e_i$, $e'_{i - 1} = 1 - e_{i - 1}$ and $e'_j = e_i$ for $j\ne i,i - 1$.
Set $b_{e,j} = (s_1^{e_1}\cdots s_{j - 1}^{e_{j - 1}})\left(\frac{1}{\alpha_j}\right)$.
If $j < i$ then obviously we have $b_{e,j} = b_{f(e),j}$.
If $j > i$ then since $s_{i - 1}^{1 - e_{i - 1}}s_i^{1 - e_i} = s_{i - 1}^{e_{i - 1}}s_i^{e_i}$, we have $b_{e,j} = b_{f(e),j}$.
If $j = i$ then
\[
b_{f(e),i} = (s_1^{e_1}\cdots s_{i - 1}^{e_{i - 1}})s_{i - 1}\left(\frac{1}{\alpha_i}\right) = -b_{f(e),i}
\]
since $s_{i - 1}^{1 - e_{i - 1}} = s_{e_{i - 1}}^{e_{i - 1}}s_{i - 1}$ and $s_{i - 1} = s_i$.
Therefore $b_{e} = \prod_{i = 1}^lb_{e,i}$ satisfies $b_{f(e)} = -b_e$.
Let $B$ be a set of complete representatives of $A/\langle f\rangle$.
Then $a^{\underline{w}}(g) = \sum_{e\in A}b_{e} = \sum_{e\in B}(b_{e} + b_{f(e)}) = 0$.
\end{proof}

The aim of this section is to prove the following theorem.
\begin{thm}\label{thm:image of 1}
For $\underline{w}\in S^l$, we have
\[
a^{\underline{w}}(g) = \frac{k_{g}^{\underline{w}}}{\prod_{\alpha\in X_g^w}\alpha}.
\]
\end{thm}

From the above lemma, we may assume $s_{i - 1}\ne s_i$ for any $i$.
By definitions, we also may assume $g\le w$, otherwise both sides are zero.

\subsection{Proof of Theorem~\ref{thm:image of 1}}
In this subsection we prove Theorem~\ref{thm:image of 1}.

We split the sum in the definition of $a^{\underline{w}}(g)$ to $e_l = 0$ part and $e_l = 1$ part.
If $e_l = 0$, then $s_1^{e_1}\cdots s_{l - 1}^{e_{l - 1}} = g$.
Hence $(s_1^{e_1}\cdots s_{l - 1}^{e_{l - 1}})\left(\frac{1}{\alpha_l}\right) = g\left(\frac{1}{\alpha_l}\right)$.
Therefore 
\[
\prod_{i = 1}^{l}(s_1^{e_1}\cdots s_{i - 1}^{e_{i - 1}})\left(\frac{1}{\alpha_i}\right) =
g\left(\frac{1}{\alpha_l}\right) \prod_{i = 1}^{l}(s_1^{e_1}\cdots s_{i - 1}^{e_{i - 1}})\left(\frac{1}{\alpha_i}\right)
\]
Similarly if $e_l = 1$ then $(s_1^{e_1}\cdots s_{l - 1}^{e_{l - 1}})\left(\frac{1}{\alpha_l}\right) = gs_l\left(\frac{1}{\alpha_l}\right) = -g\left(\frac{1}{\alpha_l}\right)$.
Therefore we have
\begin{align*}
a^{\underline{w}}(g)
& =
\frac{1}{g(\alpha_l)}\left(
\sum_{s_1^{e_1}\cdots s_{l - 1}^{e_{l - 1}} = g}\prod_{i = 1}^{l - 1}(s_1^{e_1}\cdots s_{i - 1}^{e_{i - 1}})\left(\frac{1}{\alpha_i}\right)
-
\sum_{s_1^{e_1}\cdots s_{l - 1}^{e_{l - 1}} = gs_l}\prod_{i = 1}^{l - 1}(s_1^{e_1}\cdots s_{i - 1}^{e_{i - 1}})\left(\frac{1}{\alpha_i}\right)
\right)
\\
& =
\frac{1}{g(\alpha_l)}(a^{(s_1,\ldots,s_{l - 1})}(g) - a^{(s_1,\ldots,s_{l - 1})}(gs_l)).
\end{align*}

We change the notation slightly and we get the following lemma.
\begin{lem}\label{lem:inductive formula of a}
Let $\underline{w} \in S^l$ and $u\in S$.
Then we have
\[
a^{(\underline{w},u)}(g)
=
\frac{1}{g(\alpha_u)}(a^{\underline{w}}(g) - a^{\underline{w}}(gu)).
\]
\end{lem}

To prove the theorem we need the following lemmas.

\begin{lem}\label{lem:inductive formula on X}
Let $w,g\in W$ and $u\in S$ such that $wu > w$, $sw < w$, $g,gu\le w$.
\begin{enumerate}
\item There exists unique $\beta\in X_g^{w}$ such that $s_{\beta}\in \{wg^{-1},swg^{-1}\}$.
\item There exists unique $\gamma\in X_{gu}^{w}$ such that $s_{\gamma}\in \{wug^{-1},swug^{-1}\}$.
\item We have $X_g^{w}\setminus\{\beta\} = X_{gu}^{w} \setminus\{\gamma\}$ and $X_{gu}^{wu} = X_{g}^w\cup \{\gamma\}$.
\end{enumerate}
\end{lem}
\begin{proof}
Since our Coxeter system has rank two, for $x\in W$, there exists $\alpha\in\Phi^+$ such that $s_{\alpha} = x$ if and only if $\ell(x)$ is odd.
One of elements in ${wg^{-1},swg^{-1}}$ has the odd length.
Hence there exists $\beta\in\Phi^+$ such that $s_{\beta}\in  \{wg^{-1},swg^{-1}\}$.
If $s_{\beta} = wg^{-1}$, then $s_{\beta}g = w\le w$.
If $s_{\beta} = swg^{-1}$, then $s_{\beta}g = sw\le w$.
Hence $\beta\in X_g^{w}$ and we get (1).
The proof of (2) is similar.

We prove (3).
Let $\delta\in X_{g}^w$.
Then $s_{\delta}g\le w$.
Since our Coxeter system is of rank two, if $\ell(s_{\delta}gu)\le \ell(w) - 1$, we have $s_{\delta}gu\le w$.
Hence $\delta\in X_{gu}^w$.
Therefore if $\ell(s_{\delta}g)\le \ell(w) - 2$, then since $\ell(s_{\delta}gu)\le \ell(s_{\delta}g) + 1$, we have $\delta\in X_{gu}^w$.

Let $u'$ be the element in $S$ which is not $u$.
Then we have $sw < w$, $wu' < w$.
\begin{itemize}
\item If $\ell(s_{\delta}g) = \ell(w) - 1$, then $s_{\delta}g = sw$ or $wu'$.
If $s_{\delta}g = sw$, then $s_{\delta} = swg^{-1}$, hence $\delta = \beta$.
If $s_{\delta}g = wu'$ and $w\ne u'$, the reduced expression of $wu'$ ends with $u$.
Hence $s_{\delta}gu = wu'u\le w$.
Therefore $\delta\in X_{gu}^{w}$.
If $w = u'$ then $u' = s$ since $sw < w$.
We have $\ell(s_{\delta}g) = \ell(w) - 1 = 0$, hence $s_{\delta}g = 1$.
Since $g\le w$, we have $g = u'$ or $g = 1$.
Since $\ell(s_{\delta})$ is odd, by $s_{\delta}g = 1$, we have $g = u'$ and $s_{\delta} = u' = swg^{-1}$.
Hence $\delta = \beta$.
\item If $\ell(s_{\delta}g) = \ell(w)$, then $s_{\delta}g = w$.
Hence $s_{\delta} =  wg^{-1}$.
Therefore $\delta = \beta$.
\end{itemize}
In any case, if $\delta\in X_{g}^{w}$, then $\delta = \beta$ or $\delta\in X_{gu}^w$.
Hence $X_{g}^w\setminus\{\beta\}\subset X_{gu}^w$.
If $\delta = \gamma$, the element $s_{\delta}g$ is $wu$ or $swu$.
Since $wu > w$, we have $s_{\delta}g\le w$ only when $s_{\delta}g = swu = w$.
Therefore $\delta = \beta$.
Hence $X_{g}^w\setminus\{\beta\}\subset X_{gu}^w\setminus\{\gamma\}$.
By replacing $g$ with $gu$, we get the reverse inclusion.

Since $wu > w$, for any $v\in W$, $vu\le wu$ if and only if $v\le w$ or $vu\le w$ by Property Z in \cite{MR0435249}.
Hence $X_{gu}^{wu} = X_{g}^{w}\cup X_{gu}^{w}$.
Therefore we get the last part of (3).
\end{proof}

\begin{lem}\label{lem:inductive formula on X, easy case}
Let $w,g\in W$, $u\in S$ such that $wu > w$, $sw <w$, $g\le w$ and $gu\not\le w$.
Then $X_{gu}^{wu} = X_{g}^{w}\cup \{g(\alpha_u)\}$.
\end{lem}
\begin{proof}
By Property Z in \cite{MR0435249}, for any $x\in W$, $x\le w$ implies $xu\le wu$.
Applying this to $x = s_{\gamma}g$ for $\gamma\in X_{g}^w$, we have $X_{g}^w\subset X_{gu}^{wu}$.
Since $g\le w$, we have $s_{g(\alpha_u)}g = gu\le wu$.
Therefore $g(\alpha_u)\in X_{gu}^{wu}$.
Hence $X_{g}^{w}\cup \{g(\alpha_u)\}\subset X_{gu}^{wu}$.

If $\ell(g)\le \ell(w) - 2$, then $\ell(gu)\le \ell(w) - 1$, hence $gu\le w$ since $\#S = 2$.
Therefore $\ell(g) = \ell(w) - 1$ or $\ell(w)$.
If $\ell(g) = \ell(w) - 1$, then $g = sw$ since $gu\not\le w$.
If $\ell(g) = \ell(w)$, then $g = w$.
Hence $g = w$ or $sw$.

Let $\delta\in X_{gu}^{wu}\setminus X_{g}^{w}$.
Then $s_{\delta}gu\le wu$ and $s_{\delta}g\not\le w$.
By Property Z \cite{MR0435249}, $s_{\delta}gu < s_{\delta}g$ and $s_{\delta}gu\le w$.
Therefore, from the discussion in the previous paragraph, $s_{\delta}gu = w$ or $s_{\delta}gu = sw$.
Combining $g\in \{w,sw\}$, we have $(g,s_{\delta}) = (w,wuw^{-1})$ or $(sw,swu(sw)^{-1})$.
In any case, we have $s_{\delta} = gug^{-1}$ and $\delta = g(\alpha_u)$.
\end{proof}

\begin{lem}\label{lem:inductive formula on a,k,X, not changed}
Let $\underline{w} = (s_1,\ldots,s_l)\in S^l$ such that $s_{i - 1}\ne s_i$ for any $i$ and $g\in W$.
Set $u = s_l$.
\begin{enumerate}
\item $a^{\underline{w}}(g) = a^{\underline{w}}(gu)$.
\item $k^{\underline{w}}_g = k^{\underline{w}}_{gu}$.
\item $X_{g}^{w} = X_{gu}^{w}$.
\end{enumerate}
\end{lem}
\begin{proof}
We may assume $g < gu$ by replacing $g$ with $gu$ if necessary.
We also may assume that $s_1 = s$ by swapping $s$ with $t$ if necessary.
(1) follows from Lemma~\ref{lem:inductive formula of a} and $a^{(\underline{w},u)}(g) = 0$.

For (2), first we assume $sg > g$ and $g\ne 1$.
Then the reduced expression of $g$ has a form $g = t\cdots u'$ where $u'\in S$ is the element which is not $u$, namely the reduced expression starts with $t$ and ends with $u'$.
Since $\underline{w} = (s,\ldots,u)$ and $s_{i - 1}\ne s_i$ for any $i$, we have $\ell(g)\equiv \ell(\underline{w})\pmod{2}$.
Hence the lemma follows from the definition of $k_{g}^{\underline{w}}$.
The proof in the case of $sg < g$, $g\ne 1$ is similar.

Assume $g = 1$.
If $u = s$, then $s_1 = s_l = s$, hence $\ell(\underline{w})$ is odd.
If $u = t$, then $s_1 = s$ and $s_l = t$.
Hence $\ell(\underline{w})$ is even.
In both cases, we can confirm $k_g^{\underline{w}} = k_{gu}^{\underline{w}}$ by the definition.

Since $wu < w$, by Property Z in \cite{MR0435249} we have $s_{\gamma}g \le w$ if and only if $s_{\gamma}gu\le w$.
(3) follows.
\end{proof}

\begin{proof}[Proof of Theorem~\ref{thm:image of 1}]
We prove the theorem by induction on $\ell(\underline{w})$.
If $\ell(\underline{w}) = 0$, then this is trivial.
Let $u\in S$ and we prove that the theorem is true for $(\underline{w},u)$ assuming that the theorem is true for $\underline{w}$.
If $(\underline{w},u)$ is not a reduced expression, then both sides of the theorem are zero.
Hence we may assume $(\underline{w},u)$ is a reduced expression.
By the previous lemma, we also may assume $gu > g$.
If $g\not\le w$, then by Property Z \cite{MR0435249}, $g\not\le wu$.
Hence both sides are zero.

Take $s_1,\ldots,s_l\in S$ such that $\underline{w} = (s_1,\dots,s_l)$.
If $g\le w$ and $gu\not\le w$, then $a^{\underline{w}}(gu) = 0$.
By Lemma~\ref{lem:inductive formula of a}, inductive hypothesis and Lemma~\ref{lem:inductive formula on X, easy case}, 
\[
a^{(\underline{w},u)}(g) = \frac{a^{\underline{w}}(g)}{g(\alpha_u)} = \frac{k_{g}^{\underline{w}}}{\prod_{\gamma\in X_g^w}\gamma}\frac{1}{g(\alpha_u)} = \frac{k_g^{\underline{w}}}{\prod_{\gamma\in X_{gu}^{(\underline{w},u)}}\gamma}.
\]
As in the proof of Lemma~\ref{lem:inductive formula on X, easy case}, we have $g= w$ or $g = s_1w$ (the latter does not happen when $l = 0$).
Hence $k_{g}^{\underline{w}} = k_{g}^{(\underline{w},u)} = 1$ from the definitions.
Therefore the theorem holds in this case.

We assume $g,gu\le w$.
Then $\ell(\underline{w}) > 0$.
We may assume $s_1 = s$ by swapping $(s,X)$ with $(t,Y)$ if necessary.
By Lemma~\ref{lem:inductive formula of a} and inductive hypothesis, we have
\[
a^{(\underline{w},u)}(g) = \frac{1}{g(\alpha_u)}(a^{\underline{w}}(g) - a^{\underline{w}}(gu))
=
\frac{1}{g(\alpha_u)}\left(\frac{k_g^{\underline{w}}}{\prod_{\delta\in X_g^w}\delta} - \frac{k_{gu}^{\underline{w}}}{\prod_{\delta\in X_{gu}^{w}}\delta}\right).
\]
Take $\beta,\gamma\in\Phi^+$ as in Lemma~\ref{lem:inductive formula on X}.
Then by Lemma~\ref{lem:inductive formula on X}, the right hand side is
\[
\frac{1}{\prod_{\delta\in X_{g}^w\setminus\{\beta\}}\delta}\frac{1}{\beta\gamma}\frac{1}{g(\alpha_u)}
(k_g^{\underline{w}}\gamma - k_{gu}^{\underline{w}}\delta)
=
\frac{1}{\prod_{\delta\in X_{g}^{wu}}\delta}\frac{1}{g(\alpha_u)}
(k_g^{\underline{w}}\gamma - k_{gu}^{\underline{w}}\delta).
\]
Hence it is sufficient to prove that $k_g^{\underline{w}}\gamma - k_{gu}^{\underline{w}}\delta = k_g^{(\underline{w},u)}g(\alpha_u)$.
Since $gu > g$, the reduced expression of $g$ ends with the simple reflection which is not $u$.
Hence the reduced expression of $gug^{-1}$ can be obtained by concatenating the reduced expressions of $g$, $u$ and $g^{-1}$.
Therefore we have $\ell(gug^{-1}) = \ell(g) + \ell(u) + \ell(g^{-1}) = 2\ell(g) + 1$.
Moreover, if $sg > g$, then we have $sgug^{-1} > gug^{-1}$.

First we assume $sg > g$ and $g\ne 1$.
Then $ss_{g(u)} = sgug^{-1} > gug^{-1}$.
Hence 
\[
g(u) = [\ell(g)]_{X}\alpha_s + [\ell(g) + 1]_{Y}\alpha_t
\]
by Lemma~\ref{lem:formula of root}.
Since $gu > g$ and $wu > w$, the reduced expressions of $g$ and $w$ ends with the same simple reflection.
Namely if $u'\in S$ is the element which is not $u$, then the reduced expression of $w$ is $w = s\cdots u'$ and the reduced expression of $g$ is $g = t\cdots u'$ since we assumed $sg > g$.
Since $g\le w$, the last $\ell(g)$-letters of the reduced expression of $w$ is the reduced expression of $g$.
Hence $\ell(wg^{-1}) + \ell(g) = \ell(w)$ and the reduced expression of $wg^{-1}$ starts with $s$ and ends with $s$.
Therefore $twg^{-1} > wg^{-1}$ and $s_{\beta} = wg^{-1}$.
Hence by Lemma~\ref{lem:formula of root}, we have
\begin{align*}
\beta & = \left[\frac{\ell(wg^{-1}) + 1}{2}\right]_X\alpha_s + \left[\frac{\ell(wg^{-1}) - 1}{2}\right]_Y\alpha_t\\
& = \left[\frac{\ell(w) - \ell(g) + 1}{2}\right]_X\alpha_s + \left[\frac{\ell(w) - \ell(g) - 1}{2}\right]_Y\alpha_t.
\end{align*}
A calculation of $\gamma$ is similar.
We have $\ell(wug^{-1}) = \ell(g) + \ell(u) + \ell(w^{-1})$ and the reduced expression of $wug^{-1}$ starts with $s$ and ends with $t$.
Therefore $\ell(swug^{-1}) = \ell(wug^{-1}) - 1$, $s_{\gamma} = swug^{-1}$ and $s(swug^{-1}) > swug^{-1}$.
Hence by Lemma~\ref{lem:formula of root}, we have
\[
\gamma = \left[\frac{\ell(w) + \ell(g) - 1}{2}\right]_X\alpha_s + \left[\frac{\ell(w) + \ell(g) + 1}{2}\right]_Y\alpha_t.
\]
Put $m = (\ell(w) - \ell(g) - 1)/2$ and $n = \ell(g)$.
Then we have
\begin{align*}
g(\alpha_u) & = [n]_X\alpha_s + [n + 1]_Y\alpha_t,\\
\beta & = [m + 1]_X\alpha_s + [m]_Y\alpha_t,\\
\gamma & = [m + n]_X\alpha_s + [m + n + 1]_Y\alpha_t.
\end{align*}
Therefore we have
\[
k_g^{\underline{w}}\gamma - k_{gu}^{\underline{w}}\beta
=
(k_g^{\underline{w}}[m + n]_X - k_{gu}^{\underline{w}}[m + 1]_X)\alpha_s + (k_{g}^{\underline{w}}[m + n + 1]_Y - k_{gu}^{\underline{w}}[m]_Y)\alpha_t.
\]
By the definition, we have
\begin{align*}
k_g^{\underline{w}} & = 
\sqbinom{2m + n}{m}_{\sigma^{2m + n}(X)}
=
\frac{[m + n + 1]_{\sigma^{2m + n}(X)}}{[m]_{\sigma^{2m + n}(X)}}\sqbinom{2m + n}{m - 1}_{\sigma^{2m + n}(X)}\\
& =
\frac{[m + n + 1]_{\sigma^{n}(X)}}{[m]_{\sigma^{n}(X)}}k_{gu}^{\underline{w}}.
\end{align*}
Hence, 
\begin{align*}
k_g^{\underline{w}}\gamma - k_{gu}^{\underline{w}}\beta
& =
\frac{k_{gu}^{\underline{w}}}{[m]_{\sigma^n(X)}}\left(([m + n + 1]_{\sigma^{n}(X)}[m + n]_X - [m + 1]_X[m]_{\sigma^{n}(X)})\alpha_s\right.\\
& \left.\qquad + ([m + n + 1]_{\sigma^n(X)}[m + n + 1]_Y - [m]_Y[m]_{\sigma^n(X)})\alpha_t\right).
\end{align*}
By Lemma~\ref{lem:quantum, (m + n + 1)(m + n) - (m + 1)m = n(2m + n + 1)}, this is equal to
\begin{align*}
& \frac{k_{gu}^{\underline{w}}}{[m]_{\sigma^n(X)}}([n]_X[2m + n + 1]_{\sigma^{m + 1}(X)}\alpha_s + [n + 1]_{Y}[2m + n + 1]_{\sigma^{m + 1}(X)}\alpha_t)\\
& =
\frac{k_{gu}^{\underline{w}}}{[m]_{\sigma^n(X)}}[2m + n + 1]_{\sigma^{m + 1}(X)}g(\alpha_u).
\end{align*}
Hence it is sufficient to prove
\[
k_{gu}^{\underline{w}}\frac{[2m + n + 1]_{\sigma^{m + 1}(X)}}{[m]_{\sigma^n(X)}} = k_{gu}^{(\underline{w},u)}.
\]
This follows immediately from Lemma~\ref{lem:quantim binom}.

The case of $tg > g$ is similar.
By Lemma~\ref{lem:formula of root}, we have
\[
g(\alpha_u) = [\ell(g) + 1]_X\alpha_s + [\ell(g)]_Y\alpha_t.
\]
The reduced expressions of $w$ and $g$ end the same reflection, hence $\ell(wg^{-1}) = \ell(w) - \ell(g)$.
The reduced expression of $g$ starts with $s$.
Hence the reduced expression of $wg^{-1}$ starts with $s$, ends with $t$.
Hence $s_{\beta} = swg^{-1}$, $s(swg^{-1}) > swg^{-1}$ and $\ell(s_{\beta}) = \ell(w) - \ell(g) - 1$.
Hence by Lemma~\ref{lem:formula of root}, we have
\[
\beta = \left[\frac{\ell(w) - \ell(g)}{2} - 1\right]_X\alpha_s + \left[\frac{\ell(w) - \ell(g)}{2}\right]_Y \alpha_t.
\]
We have $\ell(wug^{-1}) = \ell(w) + \ell(g) + 1$ and the reduced expression starts with $s$ and ends with $s$.
Hence $s_{\gamma} = wug^{-1}$, $ts_{\gamma} > s_{\gamma}$,  and $\ell(s_{\gamma}) = \ell(g) + \ell(w) + 1$.
Therefore by Lemma~\ref{lem:formula of root}, we have
\[
\gamma = \left[\frac{\ell(w) + \ell(g)}{2} + 1\right]_X\alpha_s + \left[\frac{\ell(w) + \ell(g)}{2}\right]_Y\alpha_t.
\]
Put $m = (\ell(w) - \ell(g))/2 - 1$ and $n = \ell(g) + 1$.
Then 
\begin{align*}
g(\alpha_u) & = [n]_X\alpha_s + [n - 1]_Y\alpha_t,\\
\beta & = [m]_X\alpha_s + [m + 1]_Y \alpha_t,\\
\gamma & = [m + n + 1]_X\alpha_s + [m + n]_Y\alpha_t.
\end{align*}
We have
\[
k_{g}^{\underline{w}} = \frac{[m + n]_{\sigma^{n}(X)}}{[m + 1]_{\sigma^{n}(X)}}k_{gu}^{\underline{w}}.
\]
Therefore, by Lemma~\ref{lem:quantum, (m + n + 1)(m + n) - (m + 1)m = n(2m + n + 1)}, we have
\begin{align*}
& k_g^{\underline{w}}\gamma - k_{gu}^{\underline{w}}\beta\\
& =
\frac{k_{gu}^{\underline{w}}}{[m + 1]_{\sigma^n(X)}}\left(([m + n]_{\sigma^{n}(X)}[m + n + 1]_X - [m]_X[m + 1]_{\sigma^{n}(X)})\alpha_s\right.\\
& \left.\qquad + ([m + n]_Y[m + n]_{\sigma^{n}(X)} - [m + 1]_Y[m + 1]_{\sigma^{n}(X)})\alpha_t\right).\\
& = \frac{k_{gu}^{\underline{w}}}{[m + 1]_{\sigma^n(X)}}([n]_{X}[2m + n + 1]_{\sigma^{m}(X)} + [n- 1]_Y[2m + n + 1]_{\sigma^m(X)})\\
& = \frac{k_{gu}^{\underline{w}}}{[m + 1]_{\sigma^n(X)}}[2m + n + 1]_{\sigma^m(X)}g(\alpha_u).
\end{align*}
Therefore it is sufficient to prove
\[
\frac{[2m + n + 1]_{\sigma^m(X)}}{[m + 1]_{\sigma^n(X)}}k_{gu}^{\underline{w}}
=
k_{gu}^{(\underline{w},u)}
\]
which is again an immediate consequence of Lemma~\ref{lem:quantim binom}.

We assume $g = 1$ and $u = t$.
Then one can check that formulas for $g(\alpha_u),\beta,\gamma$ in the case of $sg > g,g\ne 1$ hold.
Hence the theorem follow from the calculations in this case.
If $g = 1$ and $u = s$, then one can use the calculations in the case of $tg > g,g\ne 1$.
\end{proof}

\section{A homomorphism between Bott-Samelson bimodules}\label{sec:A homomorphism between Bott-Samelson bimodules}
\subsection{Finite Coxeter group of rank two and a realization}\label{subsec:Finite Coxeter group of rank two and a realization}
We add the tilde to the notation in the previous section, namely $(\widetilde{W},\widetilde{S})$ is the universal Coxeter system of rank $2$, $\widetilde{V}$ is the free $\Z[\widetilde{X},\widetilde{Y}]$-module with the action of $\widetilde{W}$, $[n]_{\widetilde{X}},[n]_{\widetilde{Y}}\in \Z[\widetilde{X},\widetilde{Y}]$ is the two-colored quantum numbers, etc.

The notation without tilde will be used for non-universal version.
Let $(W,S)$ be a Coxeter system such that $S = \{s,t\}$, $s\ne t$.
We assume that the order $m_{s,t}$ of $st$ is finite.
Let $\Coef$ be a commutative integral domain and $(V,\{\alpha_s,\alpha_t\},\{\alpha_s^\vee,\alpha_t^\vee\})$ a realization~\cite[Definition~3.1]{MR3555156}, namely $V$ is a free $\Coef$-module of finite rank with an action of $W$, $\alpha_s,\alpha_t\in V$ and $\alpha_s^\vee,\alpha_t^\vee\in \Hom_{\Coef}(V,\Coef)$ such that
\begin{itemize}
\item $\langle \alpha_s^\vee,\alpha_s\rangle = \langle \alpha_t^\vee,\alpha_t\rangle = 2$.
\item $s(v) = v - \langle \alpha_s^\vee,v\rangle\alpha_s$, $t(v) = v - \langle \alpha_t^\vee,v\rangle\alpha_t$ for any $v\in V$.
\item $\alpha_s,\alpha_t\ne 0$ and $\alpha_s^\vee,\alpha_t^\vee\colon V\to \Coef$ are surjective.
\item $[m_{s,t}]_{\widetilde{X}}(-\langle \alpha_s^\vee,\alpha_t\rangle,-\langle \alpha_t^\vee,\alpha_s\rangle) = [m_{s,t}]_{\widetilde{Y}}(-\langle \alpha_s^\vee,\alpha_t\rangle,-\langle \alpha_t^\vee,\alpha_s\rangle) = 0$.
\end{itemize}

The map $\tilde{s}\mapsto s$, $\tilde{t}\mapsto t$ gives a surjective homomorphism $\widetilde{W}\to W$.
Set $X = -\langle \alpha_s^\vee,\alpha_t\rangle$, $Y = -\langle \alpha_t^\vee,\alpha_s\rangle$.
Then $\tilde{\alpha}_s\mapsto \alpha_s$, $\tilde{\alpha}_t\mapsto \alpha_t$ gives a $\Z[\widetilde{X},\widetilde{Y}]$-module homomorphism $\widetilde{V}\to V$ which commutes with the actions of $\widetilde{W}$ where we regard $V$ as a $\Z[\widetilde{X},\widetilde{Y}]$-module via $\Z[\widetilde{X},\widetilde{Y}]\to \Coef$ defined by $\widetilde{X}\mapsto X$ and $\widetilde{Y}\mapsto Y$.
The image of $[n]_{\widetilde{X}}$ (resp.\ $[n]_{\widetilde{Y}}$) is denoted by $[n]_X$ (resp.\ $[n]_Y$).
We also have $\sqbinom{n}{m}_X,\sqbinom{n}{m}_Y\in \Coef$.

Let $R$ (resp.\ $\widetilde{R}$) be the symmetric algebra of $V$ (resp.\ $\widetilde{V}$).
We regard $R$ as a graded $\Coef$-algebra via $\deg(V) = 2$.
We put $\partial_u(p) = (p - u(p))/\alpha_u$ for $p\in R$.
The maps $\widetilde{V}\to V$ and $\Z[\widetilde{X},\widetilde{Y}]\to \Coef$ induce $\widetilde{R}\to R$.
We defined an element $\tilde{a}^{\underline{\tilde{w}}}(\tilde{g})\in \widetilde{R}[\tilde{w}(\tilde{\alpha}_{\tilde{u}})^{-1}\mid \tilde{w}\in \widetilde{W},\tilde{u}\in \widetilde{S}]$.
Let $Q$ be the field of fractions of $R$.
The image of $\tilde{a}^{\underline{\tilde{w}}}(\tilde{g})$ in $Q$ is denoted by $a^{\underline{\tilde{w}}}(\tilde{g})\in Q$.

As some of them are appeared already, objects related to the universal Coxeter system is denoted with the tilde and the corresponding letter without the tilde means the image in the finite Coxeter system.
For example, if $\underline{\tilde{w}} = (\tilde{s}_1,\tilde{s}_2,\ldots)$ is a sequence of elements in $\widetilde{S}$, then $\underline{w} = (s_1,s_2,\ldots)$ is the corresponding sequence in $S$.
As we have already explained, a sequence is denoted with the underline and removing the underline means the product of elements in the sequence.
Hence $\tilde{w} = \tilde{s}_1\tilde{s}_2\cdots\in \widetilde{W}$ and $w = s_1s_2\cdots \in W$.
For each root $\tilde{\alpha}\in \widetilde{\Phi}$, we have $\tilde{s}_{\tilde{\alpha}}\in \widetilde{W}$ and $s_{\tilde{\alpha}}\in W$.

Set $\underline{\tilde{x}} = (\tilde{s},\tilde{t},\ldots)\in S^{m_{s,t}}$ and $\underline{\tilde{y}} = (\tilde{t},\tilde{s},\ldots)\in S^{m_{s,t}}$.
The sequences $\underline{x}$ and $\underline{y}$ are the two reduced expressions of the longest element.
In general, for a sequence $\underline{w} = (s_1,s_2,\ldots,s_l)\in S^l$, we put
\[
\pi_{\underline{w}} = \prod_{i = 1}^ls_1\cdots s_{i - 1}(\alpha_{s_i})\in R.
\]
The two elements $\pi_{\underline{x}}$ and $\pi_{\underline{y}}$ are not the same in general.
By \cite[(7.9), (7.10)]{arXiv:2011.05432}, $\pi_{\underline{y}} = \pi_{\underline{x}}$ if $m_{s,t}$ is even and $\pi_{\underline{y}} = [m_{s,t} - 1]_X\pi_{\underline{x}}$ if $m_{s,t}$ is odd.
Put $\xi = [m_{s,t} - 1]_X$ if $m_{s,t}$ is even and $\xi = 1$ if $m_{s,t}$ is odd.
Then we have $\pi_{\underline{y}} = \xi [m_{s,t} - 1]_X\pi_{\underline{x}}$.
In particular, $\pi_{\underline{y}}\in \Coef^{\times}\pi_{\underline{x}}$ \cite[(6.11), (6.12)]{arXiv:2011.05432}.
The realization is even-balanced if and only if $\xi = 1$.

\begin{lem}
We have $[k]_Z[m_{s,t} - 1]_{\sigma^{k - 1}(Z)} = [m_{s,t} - k]_Z$.
\end{lem}
\begin{proof}
This follows from \cite[(6.10)]{arXiv:2011.05432}.
\end{proof}

\begin{lem}\label{lem:calculation of pi_x/x_g^w}
Let $\tilde{g}\in \widetilde{W}$ such that $\tilde{g}\le \tilde{x}$.
Then we have
\[
\frac{\prod_{\tilde{\delta}\in \widetilde{X}_{\tilde{g}}^{\tilde{x}}}\delta}{\pi_{\underline{x}}}
=
\begin{cases}
\xi\displaystyle\prod_{i = 1}^{\floor{\frac{m_{s,t} - \ell(g) - 1}{2}}}[m_{s,t} - 1]_{\sigma^{i - 1}(X)} & (\tilde{s}\tilde{g} > \tilde{g}),\\
\displaystyle\prod_{i = 1}^{\floor{\frac{m_{s,t} - \ell(g)}{2}}}[m_{s,t} - 1]_{\sigma^{i - 1}(X)} & (\tilde{s}\tilde{g} < \tilde{g}).
\end{cases}
\]
\end{lem}
\begin{proof}
We prove the lemma by backward induction on $\ell(\tilde{g})$.
If $\tilde{g} = \tilde{x}$, then by Theorem~\ref{thm:image of 1}, we have $a^{\underline{\tilde{x}}}(\tilde{x}) = (\prod_{\tilde{\delta}\in \widetilde{X}_{\tilde{x}}^{\tilde{x}}}\delta)^{-1}$.
On the other hand, for $e\in \{0,1\}^{m_{s,t}}$, we have $\underline{\tilde{x}}^{e} = \tilde{x}$ if and only if $e = (1,\ldots,1)$.
Hence by the definition of $a^{\underline{\tilde{x}}}(\tilde{x})$, we have $a^{\underline{\tilde{x}}}(\tilde{x}) = 1/\pi_{\underline{x}}$.

Next assume that $\tilde{g} = \tilde{s}\tilde{x}$.
Define $\tilde{s}_i = \tilde{s}$ if $i$ is odd and $\tilde{s}_i = \tilde{t}$ if $i$ is even.
Then $\underline{x} = (\tilde{s}_1,\ldots,\tilde{s}_{m_{s,t}})$.
For $e\in \{0,1\}^{m_{s,t}}$, $\underline{\tilde{x}}^e = \tilde{g}$ if and only if $e = (0,1,\ldots,1)$.
Hence by the definition, $a^{\underline{\tilde{x}}}(\tilde{g}) = 1/\alpha_{s_1} \prod_{i = 2}^{m_{s,t}}s_2\cdots s_{i - 1}(\alpha_{s_i})$.
Since $\underline{\tilde{y}} = (\tilde{s}_2,\tilde{s}_3,\ldots,\tilde{s}_{m_{s,t} + 1})$, we have $\pi_{\underline{y}} = \prod_{i = 2}^{m_{s,t} + 1}s_{2}\cdots s_{i - 1}(\alpha_{s_{i}}) = (1/a^{\underline{\tilde{x}}}(\tilde{g}))(s_2s_3\cdots s_{m_{s,t}}(\alpha_{s_{m_{s,t} + 1}})/\alpha_{s_1})$.
Since $\tilde{s}_1 = \tilde{s}$, $\tilde{s}\tilde{s}_{s_2\tilde{s}_3\cdots \tilde{s}_{m_{s,t}}(\tilde{\alpha}_{\tilde{s}_{m_{s,t} + 1}})} > \tilde{s}_{s_2\tilde{s}_3\cdots \tilde{s}_{m_{s,t}}(\tilde{\alpha}_{\tilde{s}_{m_{s,t} + 1}})}$.
Hence we have $s_2s_3\cdots s_{m_{s,t}}(\alpha_{s_{m_{s,t} + 1}}) = [m_{s,t} - 1]_X\alpha_{s} + [m_{s,t}]_Y\alpha_{t} = [m_{s,t} - 1]_X\alpha_{s}$ by Lemma~\ref{lem:formula of root}.
Since $s_1 = s$, we get $\pi_{\underline{y}} = [m_{s,t} - 1]_X/a^{\underline{\tilde{x}}}(\tilde{g})$.
By $\pi_{\underline{y}} = \xi[m_{s,t} - 1]_X\pi_{\underline{x}}$, we have $\xi\pi_{\underline{x}}a^{\underline{\tilde{x}}}(\tilde{g}) = 1$.
By Theorem~\ref{thm:image of 1}, the left hand side of the lemma is $(a^{\underline{\tilde{x}}}(\tilde{g})\pi_{\underline{x}})^{-1}$.
Hence we get the lemma in this case.

Assume that $\tilde{g}\ne \tilde{x},\tilde{s}\tilde{x}$.
Then there exists $\tilde{u}\in \widetilde{S}$ such that $\tilde{x}\ge \tilde{g}\tilde{u} > \tilde{g}$.
When $\tilde{g} = 1$, we take $\tilde{u} = \tilde{t}$.
\begin{itemize}
\item First assume that $\tilde{x}\tilde{u} < \tilde{x}$.
By Lemma~\ref{lem:inductive formula on a,k,X, not changed}, the left hand side is not changed if we replace $\tilde{g}$ with $\tilde{g}\tilde{u}$.
We prove that the right hand side is also not changed.
Then this gives the lemma by inductive hypothesis.
\begin{itemize}
\item Assume $\tilde{g}\ne 1$.
The reduced expression of $\tilde{x}$ is given as $\tilde{x} = \tilde{s}\cdots \tilde{u}$.
Let $\tilde{u}'\in \widetilde{S}$ be the element which is not $\tilde{u}$.
If $\tilde{s}\tilde{g} > \tilde{g}$, then the reduced expression of $\tilde{g}$ is $\tilde{g} = \tilde{t}\cdots \tilde{u}'$.
Hence $\ell(\tilde{g})\equiv \ell(\tilde{x})\pmod{2}$.
If $\tilde{s}\tilde{g} < \tilde{g}$, then the reduced expression of $\tilde{g}$ is $\tilde{g} = \tilde{s}\cdots \tilde{u}'$.
Hence $\ell(\tilde{g})\equiv \ell(\tilde{x}) + 1\pmod{2}$.
Therefore the right hand side is not changed.

\item If $\tilde{g} = 1$, then by $\tilde{x}\tilde{t} < \tilde{x}$ (recall that we took $\tilde{u} = \tilde{t}$), the reduced expression of $\tilde{x}$ is $\tilde{x} = \tilde{s}\cdots \tilde{t}$.
Hence $\ell(\tilde{x})$ is even and the right hand side is not changed.
\end{itemize}
\item Assume that $\tilde{x}\tilde{u} >\tilde{x}$.
Take $\tilde{\beta}$ and $\tilde{\gamma}$ such that $\tilde{s}_{\tilde{\beta}}\in \{\tilde{x}\tilde{g}^{-1},\tilde{s}\tilde{x}\tilde{g}^{-1}\}$ and $\tilde{s}_{\tilde{\gamma}}\in \{\tilde{x}\tilde{u}\tilde{g}^{-1},\tilde{s}\tilde{x}\tilde{u}\tilde{g}^{-1}\}$.
By Lemma~\ref{lem:inductive formula on X}, we have $(\prod_{\tilde{\delta}\in \widetilde{X}_{\tilde{g}}^{\tilde{x}}}\delta)/(\prod_{\tilde{\delta}\in \widetilde{X}_{\tilde{g}\tilde{u}}^{\tilde{x}}}\delta) = \beta/\gamma$.
We calculate $\beta/\gamma$.
We use calculations in the proof of Theorem~\ref{thm:image of 1}.
\begin{itemize}
\item 
If $\tilde{s}\tilde{g} > \tilde{g}$, $\tilde{g}\ne 1$ or $\tilde{g} = 1$, then by the proof of Theorem~\ref{thm:image of 1}, we have $\beta = [(m_{s,t} - \ell(\tilde{g}) + 1)/2]_X\alpha_s + [(m_{s,t} - \ell(\tilde{g}) - 1)/2]_Y\alpha_t$ and $\gamma = [(m_{s,t} + \ell(\tilde{g}) - 1)/2]_X\alpha_s + [(m_{s,t} + \ell(\tilde{g}) + 1)/2]_Y\alpha_t$.
Therefore by the previous lemma, we have $\gamma = [m_{s,t} - 1]_{\sigma^{(m_{s,t} - \ell(\tilde{g}) - 1)/2}(X)}\beta$.
We have $[m_{s,t} - 1]_X[m_{s,t} - 1]_Y = 1$ by \cite[(6.11),(6.12)]{arXiv:2011.05432}.
Hence we have $\beta/\gamma = [m_{s,t} - 1]_{\sigma^{(m_{s,t} - \ell(\tilde{g}) - 1)/2 - 1}(X)}$.
By inductive hypothesis, we get the lemma in this case.
\item Finally assume that $\tilde{s}\tilde{g}< \tilde{g}$.
By the proof of Theorem~\ref{thm:image of 1}, we have $\beta = [(m_{s,t} - \ell(\tilde{g}))/2 - 1]_X\alpha_s + [(m_{s,t} - \ell(\tilde{g}))/2]_Y\alpha_t$, $\gamma = [(m_{s,t} + \ell(\tilde{g}))/2 + 1]_X\alpha_s + [(m_{s,t} + \ell(\tilde{g}))/2]_Y\alpha_t$.
Hence $\gamma = [m_{s,t} - 1]_{\sigma^{(m_{s,t} - \ell(g))/2 - 2}(X)}\beta$.
Therefore $\beta = [m_{s,t} - 1]_{\sigma^{(m_{s,t} - \ell(g))/2 - 1}(X)}\gamma$ and we get the lemma.\qedhere
\end{itemize}
\end{itemize}
\end{proof}

\begin{lem}\label{lem:a in R}
Let $\underline{\tilde{w}}\in \widetilde{S}^l$.
If $0\le l \le m_{s,t}$, then $\pi_{\underline{x}}a^{\underline{\tilde{w}}}(1)\in R$.
\end{lem}
\begin{proof}
By Theorem~\ref{thm:image of 1}, the lemma follows from $\pi_{\underline{x}}/\prod_{\tilde{\gamma}\in \widetilde{X}_{1}^{\tilde{w}}}\gamma\in R$.
If $\tilde{w} = \tilde{x}$, then it follows from Lemma~\ref{lem:calculation of pi_x/x_g^w}.
By swapping $s$ with $t$, $\pi_{\underline{y}}a^{\underline{\tilde{y}}}(1)\in R$.
Since $\pi_{\underline{y}}\in \Coef^\times \pi_{\underline{x}}$, we get the lemma for $\tilde{w} = \tilde{y}$.
In general, we have $\tilde{w}\le \tilde{x}$ or $\tilde{w}\le \tilde{y}$.
If $\tilde{w}\le \tilde{x}$ then $X_1^{\tilde{w}}\subset X_{1}^{\tilde{x}}$.
Hence $\pi_{\underline{x}}/\prod_{\tilde{\gamma}\in \widetilde{X}_{1}^{\tilde{w}}}\gamma = (\pi_{\underline{x}}/\prod_{\tilde{\gamma}\in \widetilde{X}_{1}^{\tilde{x}}}\gamma)(\prod_{\tilde{\gamma}\in \widetilde{X}_{1}^{\tilde{x}}\setminus \widetilde{X}_{1}^{\tilde{w}}}\gamma)\in R$.
The same discussion implies the lemma when $\tilde{w}\le \tilde{y}$.
\end{proof}

\subsection{An assumption}
To prove the maim theorem, we need one more assumption.
In this subsection, we discuss on the assumption.
We start with the following proposition.

\begin{prop}\label{prop:the assumptions}
The following are equivalent.
\begin{enumerate}
\item $\sqbinom{m_{s,t}}{k}_Z = 0$ for any $1\le k\le m_{s,t} - 1$ and $Z\in \{X,Y\}$.
\item We have $\sqbinom{m_{s,t} - 1}{k}_{Z} = \prod_{i = 1}^{k}[m_{s,t} - 1]_{\sigma^{i - 1}(Z)}$ for $0\le k\le m_{s,t} - 1$ and $Z\in \{X,Y\}$.
\item The realization is even-balanced and $\sqbinom{m_{s,t} - 1}{k}_{Z} = \prod_{i = 1}^{k}[m_{s,t} - 1]_{\sigma^{i - 1}(Z)}$ for $0\le k\le (m_{s,t} - 1)/2$ and $Z\in \{X,Y\}$.
\end{enumerate}
\end{prop}
\begin{proof}
Assume (1).
By Lemma~\ref{lem:an inductive formula on binomial} and (1), we have $\sqbinom{m_{s,t} - 1}{k}_Z = -\sqbinom{m_{s,t} - 1}{k - 1}_Z[m_{s,t} + 1]_{\sigma^{k - 1}(Z)}$.
We have $[m_{s,t} + 1]_{\sigma^{k - 1}(Z)} =  - [m_{s,t} - 1]_{\sigma^{k - 1}(Z)}$~\cite[(6.9)]{arXiv:2011.05432}.
Hence (2) follows from induction on $k$.

Conversely assume (2) and we prove (1).
By Lemma~\ref{lem:an inductive formula on binomial}, we have
\begin{align*}
\sqbinom{m_{s,t}}{k}_Z & =
\sqbinom{m_{s,t} - 1}{k}_{\sigma^k(Z)}[k + 1]_Z - \sqbinom{m_{s,t} - 1}{k - 1}_Z[m_{s,t} -k - 1]_{\sigma^{k + 1}(Z)}\\
& = \prod_{i = 1}^k[m_{s,t} - 1]_{\sigma^{k + i - 1}(Z)}[k + 1]_Z - \prod_{i = 1}^{k - 1}[m_{s,t} - 1]_{\sigma^{i - 1}(Z)}[m - k - 1]_{\sigma^{k + 1}(Z)}.
\end{align*}
By replacing $i$ with $k - i$, we have $\prod_{i = 1}^k[m_{s,t} - 1]_{\sigma^{k + i - 1}(Z)} = \prod_{i = 0}^{k - 1}[m_{s,t} - 1]_{\sigma^{i - 1}(Z)} = [m_{s,t} - 1]_{\sigma(Z)}\prod_{i = 1}^{k - 1}[m_{s,t} - 1]_{\sigma^{i - 1}(Z)}$.
Therefore it is sufficient to prove $[m_{s,t} - 1]_{\sigma(Z)}[k + 1]_Z - [m - k - 1]_{\sigma^{k - 1}(Z)} = 0$.
By Lemma~\ref{lem:quantum, (n + 1)(m + 1)-nm=n+m+1}, we have $[m_{s,t} - 1]_{\sigma(Z)}[k + 1]_{Z} = [m_{s,t}]_{Z}[k + 2]_{\sigma(Z)} - [m_{s,t} + k + 1]_{\sigma^{k + 1}(Z)}$.
Since $[m_{s,t}]_Z = 0$ and $[m_{s,t} + k + 1]_{\sigma^{k + 1}(Z)} = -[m_{s,t} - k - 1]_{\sigma^{k + 1}(Z)}$~\cite[(6.9)]{arXiv:2011.05432}, we get (1).

We assume (2) and we prove (3).
By putting $k = m_{s,t} - 1$, we have $\prod_{i = 1}^{m_{s,t} - 1}[m_{s,t} - 1]_{\sigma^{i - 1}(Z)} = 1$.
If $m_{s,t}$ is even, by Lemma~\ref{lem:quantum num, even, odd} (\ref{lem:enum:quantum num, odd}) and $[m_{s,t} - 1]_Z^2 = 1$ \cite[(6.12)]{arXiv:2011.05432}, we get $[m_{s,t} - 1]_Z = 1$.
Hence $V$ is even-balanced and we get (3).

Assume (3) and we prove (2).
It is sufficient to prove that $\prod_{i = 1}^{k}[m_{s,t} - 1]_{\sigma^{i - 1}(Z)} = \prod_{i = 1}^{m_{s,t} - 1 - k}[m_{s,t} - 1]_{\sigma^{i - 1}(Z)}$.
By \cite[(6.11)]{arXiv:2011.05432}, since the realization is even-balanced, we have $\prod_{i = 1}^{m_{s,t} - 1}[m_{s,t} - 1]_{\sigma^{i - 1}(Z)} = 1$.
Hence the right hand side is $\prod_{i = m_{s,t} - k}^{m_{s,t} - 1}[m_{s,t} - 1]_{\sigma^{i - 1}(Z)}^{-1} =  \prod_{i = 1}^k[m_{s,t} - 1]_{\sigma^{m_{s,t} - i}(Z)}$.
Here in the last part we replaced $i$ with $m_{s,t} - i$ and used $[m_{s,t} - 1]_X[m_{s,t} - 1]_Y = 1$ \cite[(6.11), (6.12)]{arXiv:2011.05432}.
By Lemma~\ref{lem:quantum num, even, odd} (\ref{lem:enum:quantum num, shift}), we have $[m_{s,t} - 1]_{\sigma^{m_{s,t} - i}(Z)} = [m_{s,t} - 1]_{\sigma^{i - 1}(Z)}$ and we get (2).
\end{proof}

We need the following assumption to prove the main theorem.

\begin{assump}\label{assump:Assumption}
The equivalent conditions in Proposition~\ref{prop:the assumptions} hold.
\end{assump}

We have a sufficient condition of Assumption~\ref{assump:Assumption}.

\begin{prop}
If the action of $W$ on $\Coef \alpha_s + \Coef \alpha_t$ is faithful, then Assumption~\ref{assump:Assumption} holds.
\end{prop}
\begin{proof}
If $[k]_X = [k]_Y = 0$ for $1\le k\le m_{s,t} - 1$, then by \cite[before Claim 3.2, Claim 3.5]{MR3462556}, $(st)^k$ is the identity on $\Coef \alpha_s + \Coef \alpha_t$.
This is a contradiction.
Hence $[k]_X \ne 0$ or $[k]_Y\ne 0$ for any $1\le k\le m_{s,t} - 1$.
For $1\le k\le m_{s,t} - 1$, we have $[k]_X\sqbinom{m_{s,t}}{k}_X = [m_{s,t}]_X\sqbinom{m_{s,t} - 1}{k - 1}_X = 0$.
Hence if $[k]_{X}\ne 0$, then $\sqbinom{m_{s,t}}{k}_X = 0$.
Therefore if $[k]_X,[k]_Y\ne 0$ for any $1\le k\le m_{s,t} - 1$, we get the proposition.
Assume that there exists $k = 1,\ldots,m_{s,t} - 1$ such that $[k]_X = 0$.
Then $[k]_Y\ne 0$.
By Lemma~\ref{lem:quantum num, even, odd} (\ref{lem:enum:quantum num, odd}), $k$ is even and by Lemma~\ref{lem:quantum num, even, odd} (\ref{lem:enum:quantum num, even}), we have $[k]_YX = [k]_XY = 0$.
Hence $X = 0$.
Therefore by induction we have $[2n]_X = 0$ and $[2n + 1]_X = (-1)^n$ for any $n\in\Z_{\ge 0}$.
Hence $[2n + 1]_Y = [2n + 1]_X = (-1)^n\ne 0$ for any $n\in\Z_{\ge 0}$ by Lemma~\ref{lem:quantum num, even, odd} (\ref{lem:enum:quantum num, odd}).
We also have $[2n]_Y \ne 0$ if $1\le 2n\le m_{s,t} - 1$ since $[2n]_X = 0$.
Therefore for any $1\le l\le m_{s,t} - 1$, $[l]_Y\ne 0$.
Therefore $\sqbinom{m_{s,t}}{l}_Y = 0$.

Since $[2n + 1]_X \ne 0$ for any $n\in\Z_{\ge 0}$, $m_{s,t}$ is even.
Therefore if $l$ is even, $\#(2\Z\cap \{m_{s,t},\ldots,m_{s,t} - l + 1\}) = \#(2\Z\cap \{1,\ldots,l\})$.
Hence by Lemma~\ref{lem:quantum num, even, odd}, we have $\sqbinom{m_{s,t}}{l}_X = \sqbinom{m_{s,t}}{l}_Y$ which is zero as we have proved.
On the other hand, if $l$ is odd, then $[l]_X \ne 0$.
Hence $\sqbinom{m_{s,t}}{l}_X = 0$.
\end{proof}

Maybe more useful criterion is the following.
\begin{prop}\label{prop:assumption, in root system}
If the realization comes from a root datum and $W$ is the Weyl group, then Assumption~\ref{assump:Assumption} holds.
\end{prop}
\begin{proof}
We are in one of the following situation.
\begin{itemize}
\item $m_{s,t} = 2$, $\langle \alpha_s,\alpha_t^\vee\rangle = \langle \alpha_t,\alpha_s^\vee\rangle = 0$.
\item $m_{s,t} = 3$, $\langle \alpha_s,\alpha_t^\vee\rangle = \langle \alpha_t,\alpha_s^\vee\rangle = -1$.
\item $m_{s,t} = 4$, $\langle \alpha_s,\alpha_t^\vee\rangle = -1$, $\langle \alpha_t,\alpha_s^\vee\rangle = -2$.
\item $m_{s,t} = 6$, $\langle \alpha_s,\alpha_t^\vee\rangle = -1$, $\langle \alpha_t,\alpha_s^\vee\rangle = -3$.
\end{itemize}
We can check the assumption by direct calculations.
\end{proof}

The assumption is related to the existence of Jones-Wenzl projectors.
If Assumtion~\ref{assump:Assumption} holds, then $\sqbinom{m_{s,t} - 1}{k}_Z$ is invertible by \cite[(6.11), (6.12)]{arXiv:2011.05432}.
If \cite[Conjecture~6.23]{arXiv:2011.05432} is true, the assumption implies the existence of the Jones-Wenzl projector $JW_{m_{s,t} - 1}$.

\subsection{Soergel bimodules}\label{subsec:Soergel bimodules}
For a graded $R$-bimodule $M = \bigoplus_{i\in \Z}M^i$ and $k\in \Z$, we define the grading shift $M(k)$ by $M(k)^i = M^{i + k}$.

We define a category $\mathcal{C}$ as follows.
An object of $\mathcal{C}$ is $(M,(M^x_Q)_{x\in W})$ where
\begin{itemize}
\item $M$ is a graded $R$-bimodule.
\item $M^x_Q$ is a $Q$-bimodule such that $mp = x(p)m$ for $m\in M^x_Q$ and $p\in Q$.
\item $M\otimes_{R}Q = \bigoplus_{x\in W}M^x_Q$.
\item There exist only finite $x\in W$ such that $M^x_Q\ne 0$.
\item The $R$-bimodule $M$ is flat as a right $R$-module.
\end{itemize}
A morphism $(M,(M^x_Q))\to (N,(N^x_Q))$ is an $R$-bimodule homomorphism $\varphi$ of degree zero such that $(\varphi\otimes\id_Q)(M^x_Q)\subset N^x_Q$ for any $x\in W$.
Usually we denote just $M$ for $(M,(M^x_Q))$.
For $M,N\in \mathcal{C}$, we define the tensor product $M\otimes N = (M\otimes_{R}N,((M\otimes N)_Q^x))$ by $(M\otimes N)_Q^x = \bigoplus_{yz = x}M_Q^y\otimes_{Q}M_Q^z$.

Let $\mathcal{C}_Q$ be the category consisting of objects $(P^x)_{x\in W}$ where $P^x$ is a $Q$-bimodule such that $mp = x(p)m$ for $m\in P^x$, $p\in Q$ and there exists only finite $x\in W$ such that $P^x\ne 0$.
A morphism $(P_1^x)\to (P_2^x)$ in $\mathcal{C}_Q$ is $(\phi_x)_{x\in W}$ where $\phi_x\colon P_1^x\to P_2^x$ is a $Q$-bimodule homomorphism.
Obviously $M\mapsto (M^x_Q)_{x\in W}$ is a functor $\mathcal{C}\to \mathcal{C}_Q$.
We denote this functor by $M\mapsto M_Q$.
Since $M\to M\otimes_{R}Q$ is injective, this functor is faithful.
For $P_1 = (P_1^x),P_2 = (P_2^x)\in \mathcal{C}_Q$, we define $P_1\otimes P_2 = ((P_1\otimes P_2)^x)$ by $(P_1\otimes P_2)^x = \bigoplus_{yz = x}P_1^y\otimes_{Q}P_2^z$.
We have $(M\otimes N)_Q = M_Q\otimes N_Q$.

For $x\in W$, we define $Q_x\in \mathcal{C}_Q$ by
\begin{itemize}
\item $(Q_x)^x = Q$ as a left $Q$-module and the right action of $q\in Q$ is given by $m\cdot q = x(q)m$.
\item $(Q_x)^y = 0$ if $y\ne x$.
\end{itemize}
Then any object in $\mathcal{C}_Q$ is isomorphic to a direct sum of $Q_x$'s.
We have $Q_x\otimes Q_y\simeq Q_{xy}$ via $f\otimes g\mapsto fx(g)$.

Let $u\in S$ and we put $R^u = \{f\in R\mid u(f) = f\}$, $B_u = R\otimes_{R^u}R(1)$.
Then there exists a unique decomposition $B_u\otimes_{R}Q = (B_u)_Q^e\oplus (B_u)_Q^u$ as in the definition of the category $\mathcal{C}$.
Explicitly, it is given by the following.
Take $\delta_u\in V$ such that $\langle\alpha_u^\vee,\delta_u\rangle = 1$.
Then 
\begin{align*}
(B_u)_Q^e & = (\delta_u\otimes 1 - 1\otimes u(\delta_u))Q,\\
(B_u)_Q^u & = (\delta_u\otimes 1 - 1\otimes \delta_u)Q.
\end{align*}
Therefore $B_u\in \mathcal{C}$.
We have $(B_u)_Q\simeq Q_e\oplus Q_s$ and an isomorphism is given by 
\[
f\otimes g\mapsto \left(\frac{fg}{\alpha_u},\frac{fu(g)}{\alpha_u}\right).
\]
We always use this isomorphism to identify $(B_u)_Q$ with $Q_e\oplus Q_u$.

Let $M\in \mathcal{C}$ and consider $M\otimes B_u$.
Then $(M\otimes B_u)_Q\simeq M_Q\otimes_{Q}Q_e\oplus M_Q\otimes_{Q}Q_u$.
As a left $Q$-module, this is isomorphic to $M_Q\oplus M_Q$.
The right action is given by $(m_1,m_2)p = (m_1p,m_2u(p))$ for $p\in Q$.
\begin{lem}\label{lem:criterion in M}
Let $(m_1,m_2)\in M_Q\oplus M_Q$.
Then $(m_1,m_2)\in M\otimes B_u$ if and only $m_1 \alpha_u \in M$ and $m_1 - m_2\in M$.
\end{lem}
\begin{proof}
Let $m\in M$, $p_1,p_2\in R$.
Then the image of $m\otimes (p_1\otimes p_2)\in M\otimes B_u$ in $(M\otimes B_u)_Q\simeq M_Q\oplus M_Q$ is $(mp_1p_2\alpha_u^{-1},mp_1u(p_2)\alpha_u^{-1})$.
Hence $(mp_1p_2\alpha_u^{-1})\alpha_u = mp_1p_2\in M$ and $(mp_1p_2\alpha_u^{-1}) - (mp_1u(p_2)\alpha_u^{-1}) = mp_1\partial_u(p_2)\in M$.

On the other hand, assume that $m_1\alpha_u \in M$ and $m_1 - m_2\in M$.
Take $\delta_u\in V$ such that $\langle\alpha_u^\vee,\delta_u\rangle = 1$.
Then we have $u(\delta_u) = \delta_u - \alpha_u$.
Hence the image of $(m_1\alpha_u)\otimes (1\otimes 1) + (m_2 - m_1)\otimes (\delta_u\otimes 1 - 1\otimes \delta_u)\in M$ is $(m_1,m_1) + ((m_2 - m_1)(\delta_u/\alpha_u),(m_2 - m_1)(\delta_u/\alpha_u)) - ((m_2 - m_1)(\delta_u/\alpha_u),(m_2 - m_1)(u(\delta_u)/\alpha_u)) = (m_1,m_2)$.
\end{proof}

In general, for a sequence $\underline{w} = (s_1,s_2,\ldots,s_l)\in S^l$ of elements in $S$, we put $B_{\underline{w}} = B_{s_1}\otimes\cdots\otimes B_{s_l}$.
Set $b_{\underline{w}} = (1\otimes 1)\otimes\cdots\otimes(1\otimes 1)\in B_{\underline{w}}$.
The main theorem of this paper is the following.
\begin{thm}\label{thm:main theorem}
Assume Assumption~\ref{assump:Assumption}.
There exists a morphism $\varphi\colon B_{\underline{x}}\to B_{\underline{y}}$ such that $\varphi(b_{\underline{x}}) = b_{\underline{y}}$.
\end{thm}

\subsection{Localized calculus}
Since $(B_u)_Q  = (B_u)_{Q}^e\oplus (B_u)_Q^u\simeq Q_e\oplus Q_u$, for $\underline{w} = (s_1,\ldots,s_l)\in S$, we have
\[
(B_{\underline{w}})_Q\simeq \bigoplus_{e = (e_i)\in \{0,1\}^l}Q_{s_1^{e_1}}\otimes\cdots\otimes Q_{s_l^{e_l}} \simeq \bigoplus_{e\in \{0,1\}^l}Q_{\underline{w}^e}.
\]
We call the component corresponding to $e$ the $e$-component of $(B_{\underline{w}})_Q$.
As an $R$-bimodule, 
\[
B_{\underline{w}} = (R\otimes_{R^{s_1}}R)\otimes_{R}(R\otimes_{R^{s_2}}R)\otimes_{R}\cdots\otimes_{R}(R\otimes_{R^{s_l}}R)(l)
\simeq R\otimes_{R^{s_1}}R\otimes_{R^{s_2}}\cdots\otimes_{R^{s_l}}R(l).
\]
The $e$-component of $p_0\otimes p_1\otimes\cdots\otimes p_l\in R\otimes_{R^{s_1}}R\otimes_{R^{s_2}}\cdots\otimes_{R^{s_l}}R(l)$ is 
\[
\left(\prod_{i = 1}^l s_1^{e_1}\cdots s_{i - 1}^{e_{i - 1}}\left(\frac{p_{i - 1}}{\alpha_{s_i}}\right)\right)s_1^{e_1}\cdots s_l^{e_l}(p_l).
\]

We construct $\varphi\colon B_{\underline{x}}\to B_{\underline{y}}$ as follows.
First we define $\varphi_Q\colon (B_{\underline{x}})_Q\simeq \bigoplus_{e\in \{0,1\}^{m_{s,t}}}Q_{\underline{x}^e}\to \bigoplus_{f\in \{0,1\}^{m_{s,t}}}Q_{\underline{y}^{f}}\simeq (B_{\underline{y}})_Q$ explicitly and we will prove that $\varphi_Q$ satisfies $\varphi_Q(B_{\underline{x}})\subset B_{\underline{y}}$.
The definition of $\varphi_Q$ is given in \cite[2.6]{arXiv:2011.05432}.
For $\underline{w} = (s_1,\ldots,s_l)\in S^l$ and $e = (e_1,\ldots,e_l)\in \{0,1\}^l$, we put $\zeta_{\underline{w}}(e) = \prod_{i = 1}^l s_1^{e_1}\cdots s_{i - 1}^{e_{i - 1}}(\alpha_{s_i})$.
Then set
\[
G_{e}^{f} = 
\begin{cases}
\dfrac{\pi_{\underline{x}}}{\zeta_{\underline{y}}(f)} & (\underline{x}^e = \underline{y}^f),\\
0 & (\underline{x}^e \ne \underline{y}^f).
\end{cases}
\]
Now we define $\varphi_Q\colon \bigoplus_{e\in \{0,1\}^{m_{s,t}}}Q_{\underline{x}^e}\to \bigoplus_{f\in \{0,1\}^{m_{s,t}}}Q_{\underline{y}^f}$ by 
\[
\varphi_Q((q_e)) = \left(\sum_{e\in \{0,1\}^{m_{s,t}}}G_{e}^{f}q_e\right)_f = \left(\frac{\pi_{\underline{x}}}{\zeta_{\underline{y}}(f)}\sum_{\underline{x}^e =\underline{y}^f}q_e\right)_f.
\]
By the same way, we also define $\psi_Q\colon (B_{\underline{y}})_Q\to (B_{\underline{x}})_Q$.
From the definition, we have
\[
\varphi_Q(b_{\underline{x}})
=
\left(\frac{\pi_{\underline{x}}}{\zeta_{\underline{y}}(f)}\sum_{\underline{x}^e = \underline{y}^f}\prod_{i = 1}^{m_{s,t}}s_1^{e_1}\cdots s_{i - 1}^{e_{i - 1}}\left(\frac{1}{\alpha_{s_i}}\right)\right)_f.
\]
Define $r\colon W\to \widetilde{W}$ as follows.
If $w\in W$ is not the longest element, $r(w) = \tilde{s}_1\ldots \tilde{s}_l$ where $w = s_1\cdots s_l$ is the reduced expression of $w$.
If $w$ is the longest element then $r(w) = \tilde{x}$.
Then for $e\in \{0,1\}^{m_{s,t}}$, $\underline{x}^e = g$ if and only if $\underline{\tilde{x}}^e = r(g)$.
Therefore we have
\[
\varphi_Q(b_{\underline{x}})
=
\left(\frac{\pi_{\underline{x}}}{\zeta_{\underline{y}}(f)}a^{\underline{\tilde{x}}}(r(\underline{y}^f))\right).
\]

\begin{prop}\label{prop:1 -> 1}
We have $\varphi_Q(b_{\underline{x}}) = b_{\underline{y}}$ and $\psi_Q(b_{\underline{y}}) = b_{\underline{x}}$ if and only if Assumption~\ref{assump:Assumption} holds.
\end{prop}
\begin{proof}
Set $\varepsilon(f) = 1$ if $\tilde{s}r(\underline{y}^f) > r(\underline{y}^f)$ and $\varepsilon(f) = 0$ otherwise.
By Theorem~\ref{thm:image of 1} and Lemma~\ref{lem:calculation of pi_x/x_g^w}, the $f$-component of $\varphi_Q(b_{\underline{x}})$ is
\[
\frac{1}{\zeta_{\underline{y}}(f)}\sqbinom{m_{s,t} - 1}{\floor{\frac{m_{s,t} - \ell(\underline{y}^f) - \varepsilon(f)}{2}}}_{\sigma^{m_{s,t} - 1}(X)}\left(\xi^{\varepsilon(f)}\prod_{i = 1}^{\floor{\frac{m_{s,t} - \ell(\underline{y}^f) - \varepsilon(f)}{2}}}[m_{s,t} - 1]_{\sigma^{i - 1}(X)}\right)^{-1}.
\]
On the other hand, the $f$-component of $b_{\underline{y}}$ is $1/\zeta_{\underline{y}}(f)$.
Therefore $\varphi_Q(b_{\underline{x}}) = b_{\underline{y}}$ if and only if 
\begin{equation}\label{eq:condition of 1 -> 1}
\sqbinom{m_{s,t} - 1}{\floor{\frac{m_{s,t} - \ell(\underline{y}^f) - \varepsilon(f)}{2}}}_{Z} = \xi^{\varepsilon(f)}\prod_{i = 1}^{\floor{\frac{m_{s,t} - \ell(\underline{y}^f) - \varepsilon(f)}{2}}}[m_{s,t} - 1]_{\sigma^{i - 1}(Z)}.
\end{equation}
for any $f\in \{0,1\}^l$ where $Z = \sigma^{m_{s,t} - 1}(X)$.
Here we used $[m_{s,t} - 1]_{\sigma^{i - m_{s,t}}(Z)} = [m_{s,t} - 1]_{\sigma^{i - 1}(Z)}$ which follows from Lemma~\ref{lem:quantum num, even, odd} (\ref{lem:enum:quantum num, shift}).
With $Z = \sigma^{m_{s,t} - 1}(Y)$ we have another equation which is equivalent to $\psi(b_{\underline{y}}) = b_{\underline{x}}$.
Hence $\varphi(b_{\underline{x}}) = b_{\underline{y}}$ and $\psi(b_{\underline{y}}) = b_{\underline{x}}$ if and only if \eqref{eq:condition of 1 -> 1} holds for any $f\in \{0,1\}^{m_{s,t}}$ and $Z\in \{X,Y\}$.

We assume that $\varphi_Q(b_{\underline{x}}) = b_{\underline{y}}$ and $\psi_Q(b_{\underline{y}}) = b_{\underline{x}}$.
Set $f_k = (1^{m_{s,t} - k - 1},0^{k + 1})\in \{0,1\}^{m_{s,t}}$ for $0\le k\le m_{s,t} - 1$.
Then $\tilde{s}r(\underline{y}^{f_k}) > r(\underline{y}^{f_k})$ and $\ell(\underline{y}^{f_k}) = m_{s,t} - k - 1$.
Take $f = f_k$ in \eqref{eq:condition of 1 -> 1}.
Then we have $\sqbinom{m_{s,t} - 1}{\floor{k/2}}_Z = \xi\prod_{i = 1}^{\floor{k/2}}[m_{s,t} - 1]_{\sigma^{i - 1}(Z)}$.
Let $k = 0$.
Then $\xi = 1$.
Therefore $V$ is even-balanced.
Hence for any $0\le k\le m_{s,t} - 1$, we have$\sqbinom{m_{s,t} - 1}{\floor{k/2}}_Z = \prod_{i = 1}^{\floor{k/2}}[m_{s,t} - 1]_{\sigma^{i - 1}(Z)}$.
Therefore we have Assumption~\ref{assump:Assumption}.
The converse implication is easy to prove.
\end{proof}

For $\underline{\tilde{w}} = (\tilde{s}_1,\ldots,\tilde{s_l})\in S^l$ and $c = (c_1,\ldots,c_l)\in \{0,1\}^l$, we define the sequence $\underline{\tilde{w}}^{(c)}$ by removing $i$-the entry from $\underline{\tilde{w}}$ when $c_i = 0$.
For $u\in S$, we put $D_u^{(0)} = \partial_u$ and $D_u^{(1)} = u$.
\begin{lem}
Let $\underline{\tilde{w}} = (\tilde{s}_1,\ldots,\tilde{s}_l)\in \widetilde{S}^l$, $\tilde{g}\in \widetilde{W}$ and $g$ the image of $\tilde{g}$ in $W$.
For $p_1,\ldots,p_l\in R$, we have
\[
\sum_{\underline{\tilde{w}}^{e} = \tilde{g}}\prod_{i = 1}^ls_1^{e_1}\cdots s_{i - 1}^{e_{i - 1}}\left(\frac{p_i}{\alpha_{s_i}}\right)
=
\sum_{c\in \{0,1\}^l} a^{\underline{\tilde{w}}^{(c)}}(\tilde{g})g(D_{s_l}^{(c_l)}(p_lD_{s_{l - 1}}^{(c_{l - 1})}(\cdots (p_2D_{s_1}^{(c_1)}(p_1))\cdots))).
\]
\end{lem}
\begin{proof}
We prove the lemma by induction on $l = \ell(\underline{w})$.
Set $\underline{\tilde{v}} = (\tilde{s}_1,\ldots,\tilde{s}_{l - 1})$ and $p_{\underline{\tilde{v}}}^{(c)} = D_{s_{l - 1}}^{(c_{l - 1})}(p_{l - 1}D_{s_{l - 2}}^{(c_{l - 2})}(\cdots (p_2D_{s_1}^{(c_1)}(p_1))\cdots))$.
The $e_{l} = 0$ part of the left hand side in the lemma is
\[
g\left(\frac{p_{l}}{\alpha_{s_{l}}}\right) \sum_{e\in \{0,1\}^{l - 1},\underline{\tilde{v}}^{e} = \tilde{g}}\prod_{i = 1}^{l - 1}s_1^{e_1}\cdots s_{i - 1}^{e_{i - 1}}\left(\frac{p_i}{\alpha_{s_i}}\right)
=
g\left(\frac{p_{l}}{\alpha_{s_{l}}}\right)\sum_{c\in \{0,1\}^{l - 1}} a^{\underline{\tilde{v}}^{(c)}}(\tilde{g})g(p_{\underline{\tilde{v}}}^{(c)})
\]
by inductive hypothesis and similarly the $e_{l} = 1$ part is
\[
gs_{l}\left(\frac{p_{l}}{\alpha_{s_{l}}}\right)\sum_{c\in \{0,1\}^{l - 1}} a^{\underline{\tilde{v}}^{(c)}}(\tilde{g}\tilde{s_{l}})gs_{l}(p_{\underline{\tilde{v}}}^{(c)})
=
-g\left(\frac{s_{l}(p_{l})}{\alpha_{s_{l}}}\right)\sum_{c\in \{0,1\}^{l - 1}} a^{\underline{\tilde{v}}^{(c)}}(\tilde{g}\tilde{s_{l}})gs_{l}(p_{\underline{\tilde{v}}}^{(c)}).
\]
We have
\begin{align*}
& g\left(\frac{p_{l}}{\alpha_{s_{l}}}\right)a^{\underline{\tilde{v}}^{(c)}}(\tilde{g})g(p_{\underline{\tilde{v}}}^{(c)})
-g\left(\frac{s_{l}(p_{l})}{\alpha_{s_{l}}}\right)a^{\underline{v}^{(c)}}(\tilde{g}\tilde{s_{l}})gs_{l}(p_{\underline{\tilde{v}}}^{(c)})
\\
& = 
a^{\underline{\tilde{v}}^{(c)}}(\tilde{g})g\left(\frac{p_{l}p_{\underline{\tilde{v}}}^{(c)} - s_{l}(p_{l}p_{\underline{\tilde{v}}}^{(c)})}{\alpha_{s_{l}}}\right)
+
\frac{a^{\underline{\tilde{v}}^{(c)}}(\tilde{g}) - a^{\underline{\tilde{v}}^{(c)}}(\tilde{g}\tilde{s_l})}{g(\alpha_{s_{l}})}gs_{l}(p_{l}p_{\underline{\tilde{v}}}^{(c)})\\
& = 
a^{\underline{\tilde{v}}^{(c)}}(\tilde{g})g(\partial_{s_{l}}(p_{l}p_{\underline{\tilde{v}}}^{(c)}))
+
a^{(\underline{\tilde{v}}^{(c)},s_{l})}(\tilde{g})gs_{l}(p_{l}p_{\underline{\tilde{v}}}^{(c)})
&& \text{by Lemma~\ref{lem:inductive formula of a}}
\\
& = \sum_{d = 0}^1a^{\underline{w}^{(c,d)}}(\tilde{g})g(D_{s_l}^{(d)}(p_lp_{\underline{\tilde{v}}}^{(c)})).
\end{align*}
We get the lemma.
\end{proof}

Therefore we get the following.

\begin{cor}
Take $s_1,\ldots,s_{m_{s,t}}\in S$ such that $\underline{x} = (s_1,\ldots,s_{m_{s,t}})$.
For $p_1,\ldots,p_{m_{s,t}}\in R$, $\varphi_Q(p_1\otimes p_2\otimes \cdots \otimes p_{m_{s,t}}\otimes 1)$ is given by
\[
\left(
\frac{\pi_{\underline{x}}}{\zeta_{\underline{y}}(f)}
\sum_{c\in \{0,1\}^{m_{s,t}}} a^{\underline{\tilde{x}}^{(c)}}(r(\underline{y}^f))\underline{y}^{f}(D_{s_{m_{s,t}}}^{(c_{m_{s,t}})}(p_{m_{s,t}}D_{s_{m_{s,t} - 1}}^{(c_{m_{s,t} - 1})}(\cdots (p_2D_{s_1}^{(c_1)}(p_1))\cdots)))
\right)_f.
\]
\end{cor}
Hence to prove $\varphi_Q(B_{\underline{x}})\subset B_{\underline{y}}$, it is sufficient to prove that $((\pi_{\underline{x}}/\zeta_{\underline{y}}(f))a^{\underline{\tilde{x}}^{(c)}}(r(\underline{y}^f))\underline{y}^f(p))_f$ is in $B_{\underline{y}}$ for any $p\in R$.
To proceed the induction, we formulate as follows.

\begin{lem}
Assume Assumption~\ref{assump:Assumption}.
Let $p\in R$, $\underline{\tilde{w}}\in S^l$ and $\underline{\tilde{w}'}\in S^{l'}$ such that $l,l'\le m_{s,t}$.
We assume that $l < m_{s,t}$ or $(\underline{\tilde{w}'},\underline{\tilde{w}}) = (\underline{\tilde{x}},\underline{\tilde{y}})$.
Then we have
\[
\left(
\frac{\pi_{\underline{x}}}{\zeta_{\underline{w}}(f)}
a^{\underline{\tilde{w}'}}(r(\underline{w}^{f}))
\underline{w}^{f}(p)
\right)_f\in B_{\underline{w}}.
\]
\end{lem}
\begin{proof}
We prove the lemma by induction on $l$.
If $l = 0$, then the lemma means $\pi_{\underline{x}}a^{\underline{\tilde{w}'}}(1)p\in R$.
This is Lemma~\ref{lem:a in R}.

Take $\tilde{s}_1,\ldots,\tilde{s}_l\in \widetilde{S}$ such that $\underline{\tilde{w}} = (\tilde{s}_1,\ldots,\tilde{s}_l)$.
Put $a(g) = a^{\underline{\tilde{w}'}}(g)$ and $\underline{\tilde{v}} = (\tilde{s}_1,\ldots,\tilde{s}_{l - 1})$.
Then by Lemma~\ref{lem:criterion in M}, it is sufficient to prove
\begin{equation}\label{eq:in B, first condition}
\left(\left(\frac{\pi_{\underline{x}}}{\zeta_{\underline{w}}((f',0))}a(r(\underline{w}^{(f',0)}))\underline{w}^{(f',0)}(p)\right)\cdot \alpha_{s_l}\right)_{f'\in \{0,1\}^{l - 1}}\in B_{\underline{v}}
\end{equation}
and
\begin{equation}\label{eq:in B, second condition}
\left(\frac{\pi_{\underline{x}}}{\zeta_{\underline{w}}((f',0))}a(r(\underline{w}^{(f',0)}))\underline{w}^{(f',0)}(p)
-
\frac{\pi_{\underline{x}}}{\zeta_{\underline{w}}((f',1))}a(r(\underline{w}^{(f',1)}))\underline{w}^{(f',1)}(p)
\right)_{f'\in \{0,1\}^{l - 1}}
\in B_{\underline{v}}.
\end{equation}
We have
\[
\left(\frac{\pi_{\underline{x}}}{\zeta_{\underline{w}}((f',0))}a(r(\underline{w}^{(f',0)}))\underline{w}^{(f',0)}(p)\right)\cdot \alpha_{s_l}
=
\underline{v}^{f'}(\alpha_{s_l})\frac{\pi_{\underline{x}}}{\zeta_{\underline{w}}((f',0))}a(r(\underline{w}^{(f',0)}))\underline{w}^{(f',0)}(p)
\]
and by the definition of $\zeta_{\underline{w}}((f',0))$, we have $\underline{v}^{f'}(\alpha_{s_l})/\zeta_{\underline{w}}(f',0) = 1/\zeta_{\underline{v}}(f')$.
We also have $\underline{w}^{(f',0)} = \underline{v}^{f'}$.
Hence the left hand side of \eqref{eq:in B, first condition} is
\[
\left(\frac{\pi_{\underline{x}}}{\zeta_{\underline{v}}(f')}a(r(\underline{v}^{f'}))\underline{v}^{f'}(p)\right)_{f'\in \{0,1\}^{l - 1}}
\]
which is in $B_{\underline{v}}$ by inductive hypothesis.

Put $g = \underline{v}^{f'}$.
Then $\underline{w}^{(f',0)} = g$ and $\underline{w}^{(f',1)} = gs_{l}$.
Since $\zeta_{\underline{w}}((f',0)) = \zeta_{\underline{w}}((f',1)) = \underline{v}^{f'}(\alpha_{s_l})\zeta_{\underline{v}}(f')$, the $f'$-component of the left hand side of \eqref{eq:in B, second condition} is
\begin{align*}
& \frac{\pi_{\underline{x}}}{\zeta_{\underline{v}}(f')}\frac{1}{g(\alpha_{s_l})}(a(r(g))g(p) - a(r(gs_l))gs_l(p))\\
& = \frac{\pi_{\underline{x}}}{\zeta_{\underline{v}}(f')}\left(a(r(g)) g\left(\frac{p - s_l(p)}{\alpha_{s_l}}\right)
+\frac{a(r(g)) - a(r(gs_l))}{g(\alpha_{s_l})}gs_l(p)\right)\\
& = \frac{\pi_{\underline{x}}}{\zeta_{\underline{v}}(f')}\left(a(r(g))g(\partial_{s_l}(p)) +\frac{a(r(g)) - a(r(gs_l))}{g(\alpha_{s_l})}gs_l(p)\right).
\end{align*}

We prove that
\begin{equation}\label{eq:in B, target form}
\left(\frac{\pi_{\underline{x}}}{\zeta_{\underline{v}}(f')}a(r(\underline{v}^{f'}))\underline{v}^{f'}(\partial_{s_l}(p))\right)_{f'}
,\quad
\left(\frac{\pi_{\underline{x}}}{\zeta_{\underline{v}}(f')}\frac{a(r(\underline{v}^{f'})) - a(r(\underline{v}^{f'}s_l))}{\underline{v}^{f'}(\alpha_{s_l})}\underline{v}^{f'}(s_{l}(p))\right)_{f'}
\end{equation}
are in $B_{\underline{v}}$.
The first one is in $B_{\underline{v}}$ by inductive hypothesis.

For the second, we divide into two cases.
\begin{itemize}
\item First assume that $l < m_{s,t}$.
Then $\ell(\underline{v}^{f'}) + \ell(s_l) < m_{s,t}$.
Hence $r(\underline{v}^{f'}s_l) = \underline{\tilde{v}}^{f'}\tilde{s_l}$.
Therefore, by Lemma~\ref{eq:in B, target form}, we have $(a(r(\underline{v}^{f'})) - a(r(\underline{v}^{f'}s_l)))/\underline{v}^{f'}(\alpha_{s_l}) = a^{(\underline{\tilde{w}'},\tilde{s}_l)}(r(\underline{v}^{f'}))$.
Therefore if $l' < m_{s,t}$ then the second one of \eqref{eq:in B, target form} is in $B_{\underline{v}}$ by inductive hypothesis.

If $l' = m_{s,t}$, we have $\ell(\underline{\tilde{w}'},\tilde{s}_l) = m_{s,t} + 1$.
We also have $\ell(r(\underline{v}^{f'})) \le \ell(\underline{v})  = l - 1 \le m_{s,t} - 2 = \ell(\underline{\tilde{w}'},\tilde{s}_l) - 3$.
Hence $\ell(\underline{\tilde{w}'},\tilde{s}_l) - \ell(r(\underline{v}^{f'}))\ge 3$.
By Theorem~\ref{thm:image of 1} and Assumption~\ref{assump:Assumption}, $a^{(\underline{\tilde{w}'},\tilde{s}_l)}(r(\underline{v}^{f'})) = 0$.
Hence the second one of \eqref{eq:in B, target form} is zero which is in $B_{\underline{v}}$.

\item 
Next assume that $l = m_{s,t}$.
Then we have $\underline{\tilde{w}'} = \underline{\tilde{x}}$ and $\underline{\tilde{w}} = \underline{\tilde{y}}$.
In this case we prove that $a(r(\underline{v}^{f'})) = a(r(\underline{v}^{f'}s_l))$.

If $f'\ne (1,\ldots,1)$, then the calculation in the case of $l < m_{s,t}$ is still valid.
Hence $(a(r(\underline{v}^{f'})) - a(r(\underline{v}^{f'}s_l)))/\underline{v}^{f'}(\alpha_{s_l}) = a^{(\underline{\tilde{x}},\tilde{s}_l)}(r(\underline{v}^{f'}))$.
We have $\ell((\underline{\tilde{x}},\tilde{s}_l)) = m_{s,t} + 1$ and, since $f'\ne (1,\ldots,1)$, we have $\ell(r(\underline{v}^{f'}))\le m_{s,t} - 2$.
Therefore $\ell((\underline{\tilde{x}},\tilde{s}_l)) - \ell(r(\underline{v}^{f'}))\ge 3$.
By Theorem~\ref{thm:image of 1} and Assumption~\ref{assump:Assumption}, we have $a^{(\underline{\tilde{x}},\tilde{s}_l)}(r(\underline{v}^{f'})) = 0$.

We assume that $f' = (1,\ldots,1)$.
By the definition, $r(\underline{v}^{f'}s_l) = \tilde{x}$.
Hence $a^{\underline{\tilde{x}}}(\tilde{x}) = 1/\pi_{\underline{x}}$ by Theorem~\ref{thm:image of 1} and Lemma~\ref{lem:calculation of pi_x/x_g^w}.
We have $\ell(r(\underline{v}^{f'})) = m_{s,t} - 1 = \ell(\tilde{x}) - 1$.
Therefore by Theorem~\ref{thm:image of 1} and Lemma~\ref{lem:calculation of pi_x/x_g^w}, we have $a^{\underline{\tilde{x}}}(r(\underline{v}^{f'})) = 1/\pi_{\underline{x}}$ as $\xi = 1$.
\end{itemize}
We finish the proof.
\end{proof}

Theorem~\ref{thm:main theorem} is proved.

\subsection{Relation with the diagrammatic Hecke category}
Let $(W,S)$ be a general Coxeter system such that $\#S < \infty$ (we allow $\#S\ne 2$) and $(V,\{\alpha_u\}_{u\in S},\{\alpha_u^\vee\}_{u\in S})$ a realization.
We assume that for any $u_1,u_2\in S$ ($u_1\ne u_2$) such that the order $m_{u_1,u_2}$ of $u_1u_2$ is finite, we have $\sqbinom{m_{u_1,u_2}}{k}_Z = 0$ for any $Z\in \{X,Y\}$ and $1\le k\le m_{u_1,u_2} - 1$.
We can define the category $\mathcal{C},\mathcal{C}_Q$ by the same way as in \ref{subsec:Soergel bimodules}.
Let $\BSbimod$ be the full subcategory of $\mathcal{C}$ consisting of objects of a form $B_{s_1}\otimes\cdots\otimes B_{s_l}(n)$.
If $u_1,u_2\in S$, $u_1\ne u_2$ satisfies $m_{u_1,u_2} < \infty$, then we put $B_{u_1,u_2} = \overbrace{B_{u_1}\otimes B_{u_2}\otimes \cdots}^{m_{u_1,u_2}}$ and $B_{u_2,u_1} = \overbrace{B_{u_2}\otimes B_{u_1}\otimes \cdots}^{m_{u_1,u_2}}$.
By Theorem~\ref{thm:main theorem} there exists a homomorphism $\varphi_{u_1,u_2}\colon B_{u_1,u_2}\to B_{u_2,u_1}$ which sends $(1\otimes 1)\otimes (1\otimes 1)\otimes \cdots \otimes(1\otimes 1)$ to $(1\otimes 1)\otimes (1\otimes 1)\otimes \cdots \otimes(1\otimes 1)$.

Let $\mathcal{D}$ be the diagrammatic Hecke category defined by Elias-Williamson~\cite{MR3555156}.
We also assume that $\mathcal{D}$ is ``well-defined'', see \cite[5.1]{arXiv:2011.05432}.

We define a functor $\mathcal{F}\colon \mathcal{D}\to \BSbimod$ as follows.
For an object $(s_1,\ldots,s_l)\in \mathcal{D}$, we define $\mathcal{F}(s_1,\ldots,s_l) = B_{s_1}\otimes\cdots\otimes B_{s_l}$.
We define $\mathcal{F}$ on morphisms by 
\begin{align*}
\mathcal{F}\left(
\begin{tikzpicture}[baseline=0.25cm]
\draw (0,0.5cm) -- (0,0) node[circle,fill,draw,inner sep=0mm,minimum size=1.5mm] {};
\useasboundingbox (current bounding box.south west) + (-1mm,0) rectangle (current bounding box.north east) + (1mm,0);
\end{tikzpicture}
\right)
& =
(p\mapsto p\delta_u\otimes1 - p\otimes u(\delta_u)),\\
\mathcal{F}\left(
\begin{tikzpicture}[baseline=0.2cm]
\draw (0,0) -- (0,0.5cm) node[circle,fill,draw,inner sep=0mm,minimum size=1.5mm] {};
\useasboundingbox (current bounding box.south west) + (-1mm,0) rectangle (current bounding box.north east) + (1mm,0);
\end{tikzpicture}
\right)
& =
(p_1\otimes p_2\mapsto p_1p_2),\\
\mathcal{F}\left(
\begin{tikzpicture}[baseline=-0.1cm]
\draw (-0.5cm,0.5cm) arc[start angle=180, end angle=360, radius=0.5cm];
\draw (0,0) to (0,-0.5cm);
\useasboundingbox (current bounding box.south west) + (-1mm,0) rectangle (current bounding box.north east) + (1mm,0);
\end{tikzpicture}
\right)
& =
(p_1\otimes p_2\mapsto p_1\otimes 1\otimes p_2),\\
\mathcal{F}\left(
\begin{tikzpicture}[baseline=-0.1cm]
\draw (-0.5cm,-0.5cm) arc[start angle=180, end angle=0, radius=0.5cm];
\draw (0,0) to (0,0.5cm);
\useasboundingbox (current bounding box.south west) + (-1mm,0) rectangle (current bounding box.north east) + (1mm,0);
\end{tikzpicture}
\right)
& =
(p_1\otimes p_2\otimes p_3 \mapsto p_1\partial_u(p_2)\otimes p_3),\\
\mathcal{F}(\text{$2m_{u_1,u_2}$-valent vertex}) & = \varphi_{u_1,u_2}.
\end{align*}
for $u,u_1,u_2\in S$ and $p,p_1,p_2,p_3\in R$.
Here we regard $B_u\otimes B_u = R\otimes_{R^u}R\otimes_{R^u}R(2)$ and $\delta_u\in V$ is an element satisfying $\langle \alpha_u^\vee,\delta_u\rangle = 1$.

\begin{lem}
The functor $\mathcal{F}$ is well-defined.
\end{lem}
\begin{proof}
In \cite{arXiv:2011.05432}, a functor $\Lambda\colon \mathcal{D}\to \mathcal{C}_Q$ is defined and it is proved that $\Lambda$ is well-defined.
By the construction, we have $\Lambda = (\cdot)_Q\circ \mathcal{F}$.
Therefore $(\cdot)_Q\circ\mathcal{F}$ is well-defined and since $(\cdot)_Q\colon \BSbimod\to \mathcal{C}_Q$ is faithful, $\mathcal{F}$ is also well-defined.
\end{proof}

\begin{thm}
The functor $\mathcal{F}\colon \mathcal{D}\to \BSbimod$ gives an equivalence of categories.
\end{thm}
\begin{proof}
The proof is the same as that of the corresponding theorem in \cite{arXiv:1901.02336_accepted}.
It is obviously essentially surjective.
In \cite{MR3555156}, for each object $M,N\in \mathcal{D}$, elements in $\Hom_{\mathcal{D}}(M,N)$ called double leaves are defined and proved that it is a basis of $\Hom_{\mathcal{D}}(M,N)$ \cite[Theorem~6.12]{MR3555156}.
In \cite{arXiv:1901.02336_accepted}, the corresponding statement in $\BSbimod$ is proved, namely double leaves in $\Hom_{\mathcal{C}}(\mathcal{F}(M),\mathcal{F}(N))$ are defined and proved that it is a basis.
By the definition of $\mathcal{F}$, $\mathcal{F}$ sends double leaves to double leaves.
Hence $\mathcal{F}$ gives an isomorphism between morphism spaces.
\end{proof}

\end{document}